%% file: main.tex
\documentclass[runningheads]{llncs}
\usepackage{graphicx}
\usepackage[nice]{nicefrac}
\usepackage{amsfonts}
\usepackage{algorithm}
\usepackage{algpseudocode}
\usepackage{amsmath}
\usepackage{amssymb}
\usepackage{dsfont}
\usepackage{bm}
\usepackage{xcolor}
\usepackage{colortbl}
\usepackage{color}
\usepackage{graphicx}
\usepackage{multirow}
\usepackage{pifont}
\usepackage{todonotes}
\usepackage{mathtools}
\usepackage{hyperref}
\usepackage[capitalize]{cleveref}
\usepackage{threeparttable}
\usepackage{longtable}
\usepackage{caption}
\usepackage{subcaption}
\usepackage{pgfplots}

\usepackage{mathrsfs} 
\usepackage{calrsfs} 
\usepackage{dutchcal} 

\usepackage{verbatim}
\usepackage{xspace} %

\usepackage{enumerate}
\usepackage{enumitem}
\allowdisplaybreaks

\newcommand{\norm}[1]{\left\|#1\right\|}

\newcommand{\E}{\mathbb{E}}

\def\<#1,#2>{\left\langle #1,#2 \right\rangle}

\newtheorem{assumption}{Assumption}

\newcommand{\inner}[2]{\left\langle #1, #2\right\rangle}
\newcommand{\summ}[3]{\sum^{ #2 }_{ #1 } #3}
\newcommand{\fsumm}[3]{\frac{1}{M}\sum^{ #2 }_{ #1 } #3}
\newcommand{\coef}[1]{\biggr[ #1 \biggr]}

\newcommand{\mc}{\mathbcal}

\definecolor{darkgreen}{RGB}{0, 128, 0} 
\newcommand{\greencheck}{{\color{darkgreen}\checkmark}} 
\newcommand{\redcross}{{\color{red}\textbf{\texttimes}}}

\def\O{\mathcal{O}}
\begin{document}

\input{stuff/main/Title}

\maketitle

\begin{abstract}
\input{stuff/Abstract}
\end{abstract}
\section{Introduction}
\input{stuff/main/Introduction}

\section{Settings and Notation} 
\input{stuff/main/Settings}
\section{Convergence Analysis} 
\input{stuff/main/ConvergenceTheory}
\section{Discussion} \label{sec:discussion}
\input{stuff/main/Discussion}
\bibliographystyle{splncs04}
\bibliography{bib.bib}

\newpage
\appendix \label{sec:proofs}
\section{Notation}
\input{stuff/supplementary/Notation}
\section{Basic facts}
\input{stuff/supplementary/BasicFacts}
\section{Descent lemmas}
\input{stuff/supplementary/TechnicalLemmas}

\section{Variance lemmas}
\input{stuff/supplementary/OtherLemmas}
\section{Proof of Theorem~\ref{th:th_1}}
\label{sec: proof-thm-1}
\input{stuff/supplementary/ProofTheorem1}
\section{Proof of Theorem~\ref{th:th_2}}
\label{sec: proof-thm-2}
\input{stuff/supplementary/ProofTheorem2}
\section{Proof of Theorem~\ref{th:mu>0,rho>0}}
\label{sec: proof-thm-3}
\input{stuff/supplementary/ProofTheorem3}

\end{document}

%% file: stuff/main/Title.tex
\title{Local SGD for Near-Quadratic Problems: Improving Convergence under Unconstrained Noise Conditions}
\titlerunning{Local SGD for Near-Quadratic Problems}
%
\author{Anonymous authors}
\authorrunning{Anonymous authors}
\author{Andrey Sadchikov\inst{1}\and
Savelii Chezhegov\inst{1, 2, 3}\and
Aleksandr Beznosikov\inst{2,3,4} \and
Alexander Gasnikov\inst{4,1,2}
}
\authorrunning{A. Sadchikov, S. Chezhegov, A. Beznosikov, A. Gasnikov}
%
\institute{Moscow Institute of Physics and Technology, Moscow, Russia
\and
Ivannikov Institute for System Programming RAS, Moscow, Russia
\and
Sber AI Lab, Moscow, Russia
\and
Innopolis University, Innopolis, Russia
}

%% file: stuff/Abstract.tex

Distributed optimization plays an important role in modern large-scale machine learning and data processing systems by optimizing the utilization of computational resources. One of the classical and popular approaches is Local Stochastic Gradient Descent (Local SGD), characterized by multiple local updates before averaging, which is particularly useful in distributed environments to reduce communication bottlenecks and improve scalability. A typical feature of this method is the dependence on the frequency of communications. But in the case of a quadratic target function with homogeneous data distribution over all devices, the influence of frequency of communications vanishes. As a natural consequence, subsequent studies include the assumption of a Lipschitz Hessian, as this indicates the similarity of the optimized function to a quadratic one to some extent. However, in order to extend the completeness of the Local SGD theory and unlock its potential, in this paper we abandon the Lipschitz Hessian assumption by introducing a new concept of \textit{approximate quadraticity}. This assumption gives a new perspective on problems that have near quadratic properties. In addition, existing theoretical analyses of Local SGD often assume bounded variance. We, in turn, consider the unbounded noise condition, which allows us to broaden the class of studied problems.  

%% file: stuff/main/Introduction.tex
Stochastic Gradient Descent (SGD) \cite{robbins1951stochastic} is a classical and widely used algorithm for machine \cite{shalev2014understanding} and deep learning \cite{goodfellow2016deep} tasks. This is due to its easy interpretability and low computational cost. However, this approach has bottlenecks, such as sensitivity to noise, because the computation of the stochastic gradient might occur with large variance. One of the solutions to this problem is parallelization of gradient computations over local data, since collecting a larger set of stochastic realizations of the gradient leads to variance reduction. Such a method is known in the literature as Parallel SGD \cite{mangasarian1995parallel}. Now tasks require substantial computational resources due to the rapid development of artificial intelligence, making Parallel SGD advantageous over classical SGD. Moreover, in such a branch of machine learning as federated learning \cite{konevcny2016federated,kairouz2021advances,karimireddy2020scaffold}, where data privacy \cite{abadi2016deep} plays an important role, collecting information on a single device becomes basically impossible. As a consequence, parallelization of the optimization process is mandatory in these settings because of the nature of the problem.


Despite of the advantages of Parallel SGD, it has a bottleneck, which is expressed in terms of communication \cite{konevcny2016federated,verbraeken2020survey}. Indeed, in reality, the communication process can be expensive compared to the complexity of computing local stochastic gradients. As a consequence, method Local SGD is proposed by \cite{stich2018local,mcmahan2017communication}. When Parallel SGD is based on computing local stochastic gradients at a single point, Local SGD allows local steps to obtain separate point sequences for each computing device.


Local SGD has been widely studied in standard form \cite{stich2018local,khaled2020tighter}, for the variational inequalities \cite{beznosikov2020distributed,beznosikov2024similarity}, in decentralized paradigm \cite{beznosikov2022decentralized,koloskova2020unified}, and with some modifications like momentum \cite{wang2019slowmo}, quantization \cite{basu2019qsparse,reisizadeh2020fedpaq}, variance reduction \cite{liang2019variance,sharma2019parallel} and preconditioning \cite{chezhegov2024local}. In addition, new interpretations of the local technique have emerged, which have resulted in methods known in the literature as FedProx \cite{li2020federated}, SCAFFOLD \cite{karimireddy2020scaffold} and ProxSkip \cite{mishchenko2022proxskip}. Moreover, it is worth mentioning papers related to Hessian similarity \cite{shamir2014communication,hendrikx2020statistically,beznosikov2021distributed,kovalev2022optimal} of local functions, the essence of which is the local computation of one particular device. 


As said before, some interpretations of local approaches are based on the similarity of data distributions across devices. Typically, two main cases are distinguished, \textit{homogeneous} (or \textit{identical}) and \textit{heterogeneous}, that describe whether data distribution across devices are similar or different, respectively. In the Local SGD analysis this plays a huge role, as knowledge about the similarity of data allows to obtain better convergence rates, or explore certain problem structures. Thus, for example, in \cite{zhang2016parallel} it is observed that for quadratic objective functions with homogeneous data distribution the convergence rate of Local SGD is independent of communication frequency. Later, proof of this fact was provided in \cite{woodworth2020local}.
This important observation aligns with the lower bounds established by the Local AC-SA algorithm \cite{woodworth2020local,ghadimi2013optimal}, indicating that Local AC-SA is minimax optimal for quadratic objectives.


A natural development in the study of this phenomenon is the use of the assumption of the Lipschitz Hessian, since it is the condition that specifies the nature of \textit{approximate quadraticity} in the objective function. In this direction it is worth noting the papers \cite{yuan2020federated,glasgow2022sharp}. Nevertheless, the supposition of the Lipschitz Hessian is quite strong, and severely limits the class of problems considered. Therefore, our efforts are focused on finding a weaker requirement and analyzing Local SGD in the perspective of the new assumption.


It is also worth mentioning that the majority of existing research, with the exception of \cite{spiridonoff2021communication,koloskova2020unified,gorbunov2021local}, operates under the consideration of the uniformly bounded variance of stochastic gradients. Nevertheless, this assumption is unrealistic in modern scenarios. Thus, it can be easily shown that even for the case of quadratic loss functions this is not fulfilled. 

Therefore, another aim of our study is to analyze Local SGD for quadratic-like functions under some other conditions on the stochasticity. In particular, we consider the noise with unbounded variance, where the variance of stochastic gradient is proportional to the gradient norm. In the literature this assumption is called \textit{strong growth condition}, which was initially investigated under this designation by \cite{schmidt2013fast}.

To sum up, our contributions can be formulated as follows:
\begin{itemize}
    \item[(a)] We introduce the concept of \textit{approximate quadraticity}, providing a framework to gauge the proximity of various functions to a quadratic form (i.e. \cref{st:st_1}).

    \item[(b)] We establish that advancements in convergence rates for Local SGD on quadratic-like functions are achievable even in the absence of the Lipschitzness of Hessian (Theorems \ref{th:th_1}, \ref{th:th_2} and \ref{th:mu>0,rho>0}).
    
    \item[(c)] We are the first, to the best of our knowledge, to demonstrate that such advancements persist under the strong growth condition (i.e. Assumption~\ref{ass:ass_2}).

    \item[(d)] We present the best known convergence rate for Local SGD (Algorithm~\ref{alg:localsgd}) in the strongly convex and smooth case (i.e. Assumption~\ref{st:st_1}) under the strong growth condition (i.e. Assumption~\ref{ass:ass_2}).
\end{itemize}
\nopagebreak
\input{tables/Table}
\nopagebreak
\input{stuff/main/RelatedWork}
\input{stuff/main/ProblemFormulation}


%% file: tables/Table.tex
\begin{table}

\caption{Comparison of works around Local SGD}

\renewcommand{\arraystretch}{2} 

\centering
\footnotesize
\resizebox{\textwidth}{!}{
\begin{threeparttable}
\begin{tabular}{|p{1.8cm}|p{1.6cm}|p{1.9cm}|p{1.5cm}|p{1.7cm}|p{6.7cm}|}
\hline
\textbf{Reference} &
\textbf{Not-Lipschitz Hessian} &
\textbf{Unbounded Gradient} &
\textbf{Noise model} &
\textbf{Convexity} & 
\textbf{Convergence rate} \\



\hline

\multirow{2}{=}{Stich, \cite{stich2018local}} &
\multirow{2}{=}{\greencheck} &
\multirow{2}{=}{\redcross} &
\multirow{2}{=}{Uniform} 
& 
$\mu=0$ & - \\
& & & & 
$\mu > 0$ & $\O \left(
  			\frac{D^2}{K^3} + \frac{\sigma^2}{\mu M T} + 
  			\frac{\kappa G^2}{\mu K^2} \right)$ \tnote{{\color{blue}(a)}} \\

\hline

\multirow{2}{=}{Yuan, et al., \cite{yuan2020federated}} &
\multirow{2}{=}{\redcross} &
\multirow{2}{=}{\greencheck} &
\multirow{2}{=}{Uniform} 
& 
$\mu=0$ & $\tilde{\O} \left( 
            \frac{LD^2}{T} + \frac{\sigma D}{\sqrt{MT}} + \frac{Q^{\frac{1}{3}} \sigma^{\frac{2}{3}} D^{\frac{5}{3}}}{T^{\frac{1}{3}} K^{\frac{1}{3}}}
        \right)$ \tnote{{\color{blue}(b)}} \\
& & & & 
$\mu > 0$ & $\tilde{\O} \left(\text{exp. decay} + \frac{\sigma^2}{\mu M T} + \frac{Q^2                     \sigma^4}{\mu^5 T^2 K^2}   \right)$ 
                \tnote{{\color{blue}(b)}}
                \\

\hline

\multirow{2}{=}{Khaled, et al., \cite{khaled2020tighter}} &
\multirow{2}{=}{\greencheck} &
\multirow{2}{=}{\greencheck} &
\multirow{2}{=}{Uniform} 
& 
$\mu=0$ & ${\O} \left(
          \frac{D^2}{\sqrt{MT}} + \frac{\sigma^2}{L\sqrt{MT}} + \frac{\sigma^2 M}{LK}
          \right)$ \\ 
& & & & 
$\mu > 0$ & $\tilde{\O} \left( 
  			\frac{L D^2 }{T^2} + 
  			\frac{L \sigma^2}{\mu^2 M T} + \frac{L^2 \sigma^2}{\mu^3 T K} \right)$ \\

\hline

\multirow{2}{=}{Woodworth, et al., \cite{woodworth2020local}} &
\multirow{2}{=}{\greencheck} &
\multirow{2}{=}{\greencheck} &
\multirow{2}{=}{Uniform} 
& 
$\mu=0$ & $\O \left(
        \frac{L D^2}{T} 
        + \frac{\sigma D}{\sqrt{MT}} 
        + \left( \frac{ L \sigma^2 D^4}{TK} \right)^{1/3}
        \right)$ \\ 
& & & & 
$\mu > 0$ & $\tilde{\O}
        \left(
        \text{exp.decay} 
        +\frac{\sigma^2}{\mu TM}
        +\frac{L \sigma^2}{\mu^2 TK}
        \right)$  \\

\hline

\multirow{2}{=}{Spiridonoff, et al., \cite{spiridonoff2021communication}} &
\multirow{2}{=}{\greencheck} &
\multirow{2}{=}{\greencheck} &
\multirow{2}{=}{Strong growth} 
& 
$\mu=0$ & - \\
& & & & 
$\mu > 0$ & $\O \left( 
  			\frac{L^2 D^2}{\mu^2 T^2} + 
            \frac{\rho L^4\ln(TK^{-2}) D^2}{\mu^4 T^2} +
  			\frac{L \sigma^2}{\mu^2 M T} + \frac{L^2 \sigma^2}{\mu^3 T K}\right)$\tnote{{\color{blue}(c)}} \\
\hline
\multirow{4}{=}{\textbf{This work}} &
\multirow{2}{=}{\greencheck} &
\multirow{2}{=}{\greencheck} &
\multirow{2}{=}{Uniform} 
& 
$\mu=0$ & $\O \left(
        \frac{L D^2}{T} 
        + \frac{\sigma D}{\sqrt{MT}} 
        + \left( \frac{\bm{\varepsilon} L \sigma^2 D^4}{TK} \right)^{1/3}
        \right)$ 
          \tnote{{\color{blue}(d)}}
          \\
          
& & & & 
$\mu > 0$ \tnote{{\color{blue}(e)}} & $\tilde{\O}
        \left(
        \text{exp.decay} 
        +\frac{\sigma^2}{\mu TM}
        +\frac{\bm{\varepsilon} L \sigma^2}{\mu^2 TK}
        \right)$ 
        \tnote{{\color{blue}(d)}}
     \\

\cline{2-6}

\multirow{4}{=}{} &
\multirow{2}{=}{\greencheck} &
\multirow{2}{=}{\greencheck} &
\multirow{2}{=}{Strong growth} 
& 
$\mu=0$ & - \\
& & & & 
$\mu > 0$ \tnote{{\color{blue}(e)}} & $
    \tilde{\O}
    \left(
    \text{exp.decay} +
    \frac{\sigma^2}{\mu M T} + \frac{\rho L^2 \sigma^2}{\mu^3 M T^2 K}+  \frac{ \bm{\varepsilon} L \sigma^2}{\mu^2 TK}
    \right)$\tnote{{\color{blue}(c)(d)}}
        \\

\hline
\end{tabular}
\begin{tablenotes}
\item General notation: 
$\O$ omits constant factors;
$\tilde{\O}$ omits polylogarithmic and constant factors.
$D = \norm{x_0-x_*}$ -- initial distance to the minimum;
$\sigma$ -- variance of stochastic gradient; 
$\mu$ - strong convexity constant; $L$ -- Lipschitz gradient constant;
$M$ -- number of workers;
$K$ -- number of communication rounds;
$T$ -- total number of iterations;
$\kappa = \frac{L}{\mu}$ -- condition number.
For more detailed explanation see Table~\ref{tab:notation}.
\\
\tnote{{\color{blue}(a)}} $G$ represents the uniform bound for the gradient norm, i.e. $ \forall x: \norm{\nabla F(x)} \leq G$.
\\
\tnote{{\color{blue}(b)}} $Q$ signifies the Lipschitz constant of the Hessian.
\\
\tnote{{\color{blue}(c)}} $\rho$ denotes the strong growth parameter. For more details see Assumption~\ref{ass:ass_2}).
\\
\tnote{{\color{blue}(d)}} $\varepsilon$ shows how far function is from quadratic.
\\
for all functions $\varepsilon \leq 1$ and for quadratic functions $\varepsilon = 0$. For more details see~\cref{st:st_1}.
\\
\tnote{{\color{blue}(e)}} $\mu = \mu_Q + \mu_R$, where $Q$ and $R$ as such convex functions that $Q + R = F$. For more details see \cref{st:st_1}.
\end{tablenotes}
\end{threeparttable}
}
\label{tab:different_row_counts}
\end{table}

%% file: stuff/main/RelatedWork.tex
\section{Related work}

One of the first results of the Local SGD study is \cite{stich2018local}, where the convergence for the strongly convex case was analyzed under condition of bounded gradients, which appears to be quite restrictive. The research is continued by \cite{khaled2020tighter} and \cite{woodworth2020local}, which get rid of the bounded gradient norms assumption and provides the best known rates for convex and $L$-smooth objectives under consideration of uniformly bounded noise model. Moreover, \cite{khaled2020tighter} and \cite{woodworth2020local} show a partially unimprovable result \cite{glasgow2022sharp} for both the identical and heterogeneous case in different convexity considerations. The work \cite{spiridonoff2021communication} generalizes previous ones, replacing uniformly bounded variance assumption with strong growth condition, that captures wider range of problems.  

Regarding the Lipschitz Hessian assumption, one of the main papers is \cite{yuan2020federated} where study is done in the convex and strongly convex cases. It is worth saying that we need to suppose that the data distribution between devices is identical in order to get the benefit from optimization of quadratic objective functions. We provide these bounds in Table~\ref{tab:different_row_counts}.

Our work achieves similar results in terms of the bounds obtained by \cite{woodworth2020local} and outperforms upper bounds from \cite{spiridonoff2021communication}. For more details see \cref{tab:different_row_counts}.

%% file: stuff/main/ProblemFormulation.tex
\section{Problem formulation} \label{subsec:problem_formulation}

Let us formally introduce the problem we solve. Consider a scenario where \( M \) devices \( 1, 2, \ldots, M \), collectively solve an optimization problem, aiming to find:
\begin{align}
\label{problem_main}
    x_* := \arg\min_{x \in \mathbb{R}^d} 
\left( F(x) = \frac{1}{M} \sum_{m=1}^{M} F_m(x) \right),
\end{align}
where $F_m(x) = \mathbb{E}_{z_m\sim\mathcal{D}_m}[F_m(x, z_m)] $ is the loss function for the $m^{th}$-client. Here, $\mathcal{D}_m$ denotes data distribution for the $m^{th}$-client and \(x\) denotes parameters of the model.

Since the consideration of \textit{approximate quadraticity} requires a homogeneous distribution of data across devices, $\mathcal{D}_m$ can be replaced by $\mathcal{D}$, which is some distribution of data characterizing each client. As a consequence, we get
\begin{align}
    \label{eq:2904_1}
    F_1(x) = \ldots = F_M(x) = F(x),
\end{align}
where $F(x)$ is defined as follows:
\[ F(x) := \mathbb{E}_{z \sim \mathcal{D}} [F(x, z)].\]

Now we are ready to get familiar with the formal definition of Local SGD.
\input{algorithms/LocalSgd}
The algorithm is constructed as follows. Let there be $M$ clients communicating only with the server. At each iteration, the $m^{th}$-client computes the local stochastic gradient $\mc{g}_t^m$ and makes an SGD step (line 10). If there is a communication moment (see lines 7-8), each node forwards the current point to the server, which averages it and sends it back to each client. The number of local steps is denoted by $H$, the number of communication rounds by $K$. Then the total number of iterations of the algorithm is $T = KH$. 



It is important to highlight how we measure complexity in distributed optimization. Unlike classical optimization, where complexity is denoted as the number of times the oracle is called, i.e., computing local gradients $\nabla F(x^m, z^m)$, we also consider the communication complexity. Generally, we define it as the quantity of information we transmit or the number of communications in our case.
Thus, in this work, we emphasize reducing the number of communication rounds $K$ or increasing the number of SGD steps between exchanging the information, denoted as $H$, while keeping the total number of SGD steps $T$ unchanged.

%% file: algorithms/LocalSgd.tex
\begin{algorithm}
    \caption{Local SGD}
    \label{alg:localsgd}
    \begin{algorithmic}[1]
        \State \textbf{Input:} 
        Initial vector $x_0 = x^m_0$ for all $m \in \{1, \ldots, M\}$; 
        stepsize $\gamma > 0$; $K$ -- number of communication rounds; $H$ -- number of local steps; $T = KH$ -- number of total iterations;
        \For{$t = 0, \ldots, T - 1$}
            \For{$m = 1, \ldots, M$ in parallel}
                \State Sample $z^m_t \sim \mathcal{D}$
                \State Evaluate stochastic gradient $\mc{g}_t^m = \nabla F(x_t^m, z^m_t)$
            \EndFor
            \If{$t+1 \mod H = 0$}
                \State $x^m_{t+1} = \frac{1}{M} \sum^M_{j = 1} (x^j_t - \gamma \mc{g}^j_t) $
            \Else
                \State $x^m_{t+1} = x_t^m - \gamma \mc{g}_t^m$ \label{eq:upd_eq}
            \EndIf
        \EndFor
    \end{algorithmic}
\end{algorithm}

%% file: stuff/main/Settings.tex
In this section, we introduce the set of assumptions under which we explore Local SGD.
\begin{definition}[Strong convexity and smoothness]\label{def:def_1} 
    The function $f$ is $\mu$-strongly convex and $L$-smooth, if  $\forall x, y \in \mathbb{R}^d$,
    \[
    \frac{\mu}{2} \norm{x - y}^2 \leq f(y) - f(x) - \inner{\nabla f(x)}{y-x} \leq \frac{L}{2} \norm {x-y}^2.
    \]
\end{definition}
\begin{corollary} \label{cor:nesterov}
Suppose that the function $f$ is $0$-strongly convex and $L$-smooth. Then
    \[
    \frac{1}{2L} \norm{\nabla f(x) - \nabla f(y)}^2 \leq f(y) - f(x) - \inner{\nabla f(x)}{y-x}.
    \]
\end{corollary}
After the definition above, we introduce the assumption about \textit{approximate quadraticity}.
\begin{assumption}[{$\varepsilon$ - decomposition}]\label{st:st_1}
    Assume that the objective function can be decomposed as $F = Q + R$, where $Q$ is a quadratic $\mu_Q$-strongly convex and $L_Q$-smooth function, $R$ is a $\mu_R$-strongly convex and $L_R$-smooth function, and $F$ is a $L$-smooth function.
\end{assumption}
Under \cref{st:st_1} we denote $\varepsilon := \frac{L_R}{L}$. We also introduce the sum of $\mu_Q$ and $\mu_R$ as $\mu$, and $\frac{L}{\mu}$ as $\kappa$.
\begin{corollary} \label{cor:linearity}
    Suppose that Assumption~\ref{st:st_1} holds. Then the following takes place:
    \begin{enumerate}
    \item[a)] $\nabla Q$ is a linear function;
    \item[b)] \(\varepsilon \leq 1\);
    \item[c)] $F$ is $\mu$-strongly convex function.
\end{enumerate}
\end{corollary}
In order to show theoretical convergence guarantees, we introduce assumption on the stochastic oracle: unbiasness and strong growth condition.
\begin{assumption}[Stochastic oracle] \label{ass:ass_2}
    There exist constants $\sigma$ and $\rho$ such that:
    \begin{align*}
     \hspace{-5cm}\E_{z \sim \mathcal{D}}[\nabla F(x, z)] = F(x), \qquad &(\textit{unbiased oracle})\\
     \E_{z \sim \mathcal{D}} \norm {\nabla F(x, z) - \nabla F(x)}^2 \leq \sigma^2
    + \rho \norm{\nabla F(x)}^2. \qquad &(\textit{strong growth condition})
    \end{align*}
\end{assumption}



%% file: stuff/main/ConvergenceTheory.tex
In this section, we present the convergence results of \cref{alg:localsgd} with different settings. First, we start with the investigation of the strongly convex case under bounded variance assumption.
\begin{theorem}[Bounded variance with $\mu > 0$] \label{th:th_1}
    Suppose that Assumptions~\ref{st:st_1} and \ref{ass:ass_2} hold with $\mu > 0$ and $\rho = 0$. Then after $T = KH$ iterations of \cref{alg:localsgd}, if $T \leq 2\kappa$, we choose $\gamma_t = \frac{1}{6L}$ and have
    \begin{align*}
        \E [F(\tilde{x}_T) - F(x_*)] = {\O}
        \left(
        L\exp{\left(\frac{-\mu T}{L}\right)} \norm{x_{0} - x_*}^2
        +\frac{\sigma^2}{\mu TM}
        +\frac{\varepsilon L \sigma^2}{\mu^2 TK}
        \right).
    \end{align*}
    On the other hand, for the case $T > 2\kappa$ we choose the sequence $\{\gamma_t\}_{t=0}$ in the following way:
    \begin{align*}
    \gamma_t=
        \begin{cases}
             \frac{1}{6L}, \qquad \qquad &\text{if } t \leq \frac{T}{2},
            \\
            \frac{2}{\mu(\xi + t)}, \qquad \qquad &\text{if } t > \frac{T}{2},
        \end{cases}
    \end{align*}    
with $\xi = \frac{12 L}{\mu} - \frac{T}{2}$. For this choice of the steps, we get
    \begin{align*}
        \E [F(\tilde{x}_T) - F(x_*)]
        =
        \tilde{\O}
        \left(
        L\exp{\left(\frac{-\mu T}{L}\right)} \norm{x_{0} - x_*}^2
        +\frac{\sigma^2}{\mu TM}
        +\frac{\varepsilon L \sigma^2}{\mu^2 TK}
        \right).
    \end{align*}
\end{theorem}
\begin{remark}
    Hereafter the point $\tilde{x}_T$ denotes the weighted average of the points $\{\bar{x}_t\}_{t=0}^{T-1}$. What is more, the weights differs from one theorem to another. For more details see Appendix \ref{sec: proof-thm-1}, \ref{sec: proof-thm-2} and \ref{sec: proof-thm-3}.
\end{remark}
\begin{remark}
\label{rem: 2}
    This result is correlated with \cite{yuan2020federated}, where the Lipschitzess of the Hessian is assumed. Indeed, if $\varepsilon = 0$, then $Q = 0$ (see \cref{tab:different_row_counts}). However, we can see an improvement in terms of power of $D$ in the final bound: $\nicefrac{4}{3}$ in our paper instead of $\nicefrac{5}{3}$. Moreover, this result is consistent with the analysis from \cite{woodworth2020local}
    except the last term, which takes into account the frequency of communications. For the quadratic problem (i.e., when $\varepsilon = 0$) there is no dependence on $H$, then the last summand goes to zero, whereas \cite{woodworth2020local} does not consider such an effect. 
\end{remark} 
\noindent Next, we introduce the convergence result for the convex case under bounded variance.
\begin{theorem}[Bounded variance with $\mu = 0$] \label{th:th_2}
    Suppose that Assumptions \ref{st:st_1} and \ref{ass:ass_2} hold with $\mu = 0, \rho = 0$. Then after $T = KH$ iterations of \cref{alg:localsgd}, with the choice of stepsizes as 
    \begin{align*}
        \gamma_t = \gamma = \begin{cases}
            \min \left\{  \frac{1}{6L}, \frac{\norm{r_0} \sqrt{M}}{\sigma \sqrt{T}}  \right\}, \qquad \qquad &\text{ if } H = 1 \text{ or } M = 1, \\
            \min \left\{ \frac{1}{6L}, \frac{\norm{r_0} \sqrt{M}}{\sigma \sqrt{T}}, \left(\frac{\norm{r_0}^2}{L_R \sigma^2TH} \right)^{1/3} \right\}, \qquad \qquad&\text { else, }
        \end{cases}
    \end{align*}
    we gain:
    \begin{align*}
        \E [F(\tilde{x}_T)] - F(x_*)
        = \O \left(
        \frac{L \norm{r_0}^2}{T} 
        + \frac{\sigma \norm{r_0}}{\sqrt{MT}} 
        + \left( \frac{\varepsilon L \sigma^2 \norm{r_0}^4}{TK} \right)^{1/3}
        \right).
    \end{align*}
\end{theorem}
\begin{remark}
This result is also consistent with the analysis from \cite{woodworth2020local}, except for the last term. The description of the difference is the same as in \cref{rem: 2}. Nevertheless, our bound is worse than in \cite{yuan2020federated}, if we compare the terms related to approximate quadraticity. This is due to different assumptions on the stochastic oracle; the \cite{yuan2020federated} considers a $4^{th}$-bounded central moment of the stochastic oracle, whereas we get $2^{th}$-bounded central moment when $\rho = 0$.    
\end{remark}
\noindent Finally, we present the theoretical guarantees of convergence for the novel case:  strongly convex case under unbounded variance for the near-quadratic functions.
\begin{theorem}[Unbounded variance with $\mu > 0$] \label{th:mu>0,rho>0}
    Suppose Assumptions \ref{st:st_1} and \ref{ass:ass_2} hold for $\mu > 0$. Then, after $T = KH$ iterations of \cref{alg:localsgd} with
    \begin{align*}
        \gamma_t = \gamma = \min\left\{\frac{\mu}{3\rho L^2}, \frac{1}{6L}, \frac{\sqrt{\mu}}{\sqrt{6CH\rho L^2}}, \frac{\ln\left(\max\left\{2, \mu^2\norm{r_0}^2T^2M/\sigma^2\right\}\right)}{2\mu T}\right\},
    \end{align*}
    where $C = \mu + L_R$, we obtain
    \begin{align*}
        &\E [F(\tilde{x}_t) - F(x_*)] =\\ &\tilde{\O}\Bigg(\norm{r_0}^2\max\left\{\frac{\rho L^2}{\mu}, L, \frac{\sqrt{CH\rho L^2}}{\sqrt{\mu}}\right\}\exp\left(-\mu T \min\left\{\frac{\mu}{\rho L^2}, \frac{1}{L}, \frac{\sqrt{\mu}}{\sqrt{CH\rho L^2}}\right\}\right) \\&\hspace{6.6cm}+ \frac{\sigma^2}{\mu M T} + \frac{\rho L^2H\sigma^2}{M\mu^3T^3}+  \frac{\varepsilon LH\sigma^2}{\mu^2 T^2}\Bigg).
    \end{align*}
\end{theorem}
\begin{remark}
    Interpretation of Theorem \ref{th:mu>0,rho>0} is also correct. Indeed, according to the bound above, $\nicefrac{H}{T^3}$ can be bounded as $\nicefrac{1}{T^2}$ for the third term, what leads us the independence from $H$ with the same rate, as in existing results \cite{spiridonoff2021communication}. 
\end{remark}
\begin{remark}
    This result outperforms the existing one \cite{spiridonoff2021communication}. In the absence of stochasticity, we have exponential decreasing instead of $\mathcal{O}\left(\nicefrac{1}{T^2}\right)$, and the other summands are better in terms of multiplicative factors, in particular $\nicefrac{L}{\mu}$.
\end{remark}

%% file: stuff/main/Discussion.tex
In our paper, we introduce such a notion as \textit{approximate quadraticity} through the assumption of $\varepsilon$-decomposition (see \cref{st:st_1}). It allows us to characterize the similarity of the optimized function to a quadratic one. Naturally, several questions arise:
is \cref{st:st_1} comparable to reality? How does the parameter $\varepsilon$ reflect the near-quadratic behavior of the target functions? To understand this, let us consider expressing $\E[\|\bar{x}_{t+1} - x_*\|^2]$ in terms of $\E[\|\bar{x}_{t} - x_*\|^2]$ instead of solving a recurrent relation and analyzing $\E[\|\bar{x}_T - x_*\|^2]$ or $\E[F(\bar{x}_T) - F(x_*)]$. As can be observed from Theorem~\ref{th:th_1} (see Appendix \ref{eq:interesting}), we have:
\begin{align}
    \E \|\bar{x}_{t+1}-x_*\|^2
    &\leq
    (1 - \gamma \lambda) \E \|\bar{x}_t - x_*\|^2 
    + \frac{\gamma^2 \sigma^2}{M}
    + 2 \varepsilon L (H-1) \gamma^3 \sigma^2. \label{eq:recurrent}
\end{align}
In fact, the parameter $\varepsilon$ can vary from the current point of the iterative process \cref{alg:localsgd}, or more precisely:
$$\varepsilon \sim \frac{1}{M}\sum\limits_{m=1}^M \varepsilon(x_t^m).$$
Therefore, considering $\varepsilon$ as a function of some set of points $\mathbb{S} \subset \mathbb{R}^d$, we can observe that for functions with a Lipschitz Hessian, we know that the decay rate of $\varepsilon$ is high. The following graphs illustrate the decay rate of $\varepsilon(x)$ (as a maximum $\varepsilon(y)$ for all $\norm{y} \leq \norm{x}$)  for some functions:
\input{graphs/EpsGraphs}
In the graph (b), by LogLoss with $\ell_2$ regularization, we mean 
\begin{equation} \label{eq:logloss}
    y = -\ln\left(\frac{1}{1 + e^{-x}} \right) + 0.03 x^2.
\end{equation}

And in the graph (c), we analyze a well-known function that yields lower bound estimates for Local SGD \cite{glasgow2022sharp}:
\begin{equation} \label{eq:piecewise}
    \mathcal{F}(x) = \begin{cases} 
      x^2 / 5 & \text{if } x < 0; \\
      x^2 & \text{if } x \geq 0. \\
   \end{cases}
\end{equation}

Here, we notice that for functions (a) and (b), $\varepsilon$ decreases rapidly (because they satisfy the Lipschitz Hessian assumption). Due to the decay of $\varepsilon$ we can increase number of local steps $H$ or reduce number of communcations $R$, which means enhancing communcation complexity.


However, for function (c), $\varepsilon$ remains constant in any neighborhood of its optimum. This may provide additional comprehension into why $\mathcal{F}$ yields lower bounds for Local SGD.


From this example, we can observe that~\eqref{eq:recurrent} provides valuable insights into the convergence rate for different types of functions. However, abandoning the assumption of a Lipschitz Hessian means that we cannot estimate the decay rate of $\varepsilon$. Therefore, conducting a meaningful asymptotic analysis while maintaining generality seems impossible.

Nevertheless, it remains an open question for future research, and more specifically, the construction of a theory that takes into account the dependence of the parameter $\varepsilon$ on the current point $x$.




\subsection*{Acknowledgements}

The work of A. Gasnikov was supported by a grant for research centers in the field of artificial intelligence, provided by the Analytical Center for the Government of the Russian Federation in accordance with the subsidy agreement (agreement identifier 000000D730321P5Q0002) and the agreement with the Ivannikov Institute for System Programming of the Russian Academy of Sciences dated November 2, 2021 no. 70-2021-00142.

%% file: graphs/EpsGraphs.tex
\begin{figure}[H]
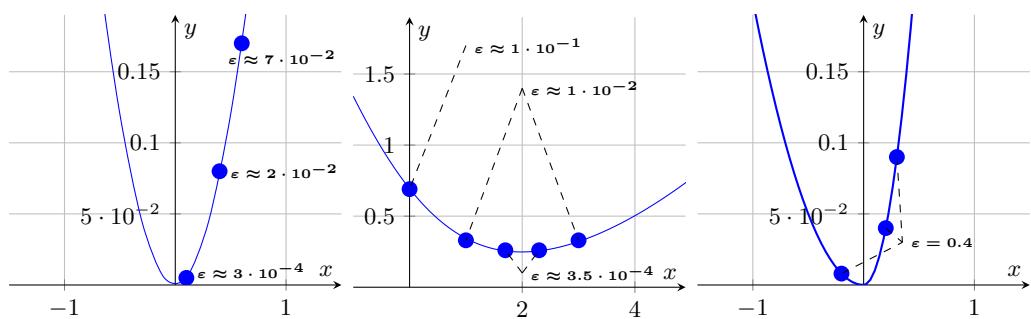

\centering
\begin{subfigure}{0.25\textwidth}
\include{graphs/LosCoshLossGraph}
\caption{\scriptsize LogCoshLoss: $y = \ln(\cosh(x))$ }
\end{subfigure}
\hfill
\begin{subfigure}{0.25\textwidth}
\include{graphs/LogLossGraph}
\caption{\scriptsize LogLoss with $\ell_2$ reg.  ~\eqref{eq:logloss}}
\end{subfigure}
\hfill
\begin{subfigure}{0.25\textwidth}
\include{graphs/PiecewiseQuadraticGraph}
\caption{\scriptsize Lower-bound function $\mathcal{F} $ ~\eqref{eq:piecewise}}
\end{subfigure}

\caption{Decay of $\varepsilon$ when getting closer to minima}
\label{graph:eps_graphs}
\end{figure}

%% file: graphs/LosCoshLossGraph.tex
\begin{tikzpicture}
\begin{axis}[
    width=6cm,
    xlabel={$x$},
    ylabel={$y$},
    xmin=-1.5, xmax=1.5,
    ymin=0, ymax=0.19, 
    axis lines=center,
    grid=both,
    domain=-10:10,
    samples=100,
]

\addplot[blue, smooth] {ln(cosh(x))};

\coordinate (A1) at (axis cs:0.6,0.17);
\coordinate (B1) at (axis cs:0.4,0.16);

\fill[blue, thick] (A1) circle (3pt);
\node[right] at (B1) {\tiny $\bm{\varepsilon \approx 7 \cdot 10^{-2}}$};

\coordinate (A2) at (axis cs:0.4,0.08);
\coordinate (B2) at (axis cs:0.42,0.08);

\fill[blue, thick] (A2) circle (3pt);
\node[right] at (B2) {\tiny $\bm{\varepsilon \approx 2 \cdot 10^{-2}}$};



\coordinate (A4) at (axis cs:0.1,0.005);
\coordinate (B4) at (axis cs:0.12,0.01);

\fill[blue, thick] (A4) circle (3pt);
\node[right] at (B4) {\tiny $\bm{\varepsilon \approx 3 \cdot 10^{-4}}$};


\end{axis}
\end{tikzpicture}

%% file: graphs/LogLossGraph.tex
\begin{tikzpicture}
\begin{axis}[
    width=6cm,
    xlabel={$x$},
    ylabel={$y$},
    xmin=-1, xmax=4.9,
    ymin=-0, ymax=1.9,
    axis lines=center,
    grid=both,
    samples=100
]
\addplot[blue, domain=-1:5] {-ln(1/(1 + exp(-x))) + x^2/33};

\coordinate (A1) at (axis cs:1, 0.33);
\coordinate (B1) at (axis cs:2, 1.4);
\coordinate (C1) at (axis cs:3,0.33);

\fill[blue, thick] (A1) circle (3pt);
\fill[blue, thick] (C1) circle (3pt);
\node[right] at (B1) {\tiny $\bm{\varepsilon \approx 1 \cdot 10^{-2}}$};

\draw[dashed] (B1) -- (A1);
\draw[dashed] (B1) -- (C1);
\coordinate (A2) at (axis cs:1.7, 0.26);
\coordinate (B2) at (axis cs:2, 0.1);
\coordinate (C2) at (axis cs:2.3,  0.26);

\fill[blue, thick] (A2) circle (3pt);
\fill[blue, thick] (C2) circle (3pt);
\node[right] at (B2) {\tiny $ \bm{\varepsilon \approx 3.5 \cdot 10^{-4}}$};

\draw[dashed] (B2) -- (A2);
\draw[dashed] (B2) -- (C2);
 \coordinate (A2) at (axis cs:0, 0.69);
\coordinate (B2) at (axis cs:1, 1.7);

\fill[blue, thick] (A2) circle (3pt);
\node[right] at (B2) {\tiny $\bm{\varepsilon \approx 1 \cdot 10^{-1}}$};

\draw[dashed] (B2) -- (A2);
\end{axis}
\end{tikzpicture}

%% file: graphs/PiecewiseQuadraticGraph.tex
\begin{tikzpicture}[
  declare function={
    func(\x)= (\x < 0) * (x*x/5)   +
              (\x >= 0) * (x*x)
   ;
  }
]
\begin{axis}[
    width=6cm,
    xlabel={$x$},
    ylabel={$y$},
    xmin=-1.5, xmax=1.5,
    ymin=0, ymax=0.19, 
    axis lines=center,
    grid=both,
    domain=-2:2,
    samples=100,
]

\coordinate (A1) at (axis cs:-0.2,0.008);
\coordinate (A2) at (axis cs:0.2,0.04);
\coordinate (A3) at (axis cs:0.3,0.09);

\coordinate (B1) at (axis cs:0.35,0.03);

\fill[blue, thick] (A1) circle (3pt);
\fill[blue, thick] (A2) circle (3pt);
\fill[blue, thick] (A3) circle (3pt);

\node[right] at (B1) {\tiny $\bm{\varepsilon = 0.4}$};

\draw[dashed] (A1) -- (B1);
\draw[dashed] (A2) -- (B1);
\draw[dashed] (A3) -- (B1);

\addplot [blue,thick] {func(x)};
\end{axis}
\end{tikzpicture}

%% file: stuff/supplementary/Notation.tex
To begin with, let us introduce useful contractions. 
By the capital Latin letters \( F, \ Q, \ R \) we denote functions. 
Corresponding stochastic gradients are represented by the bold Gothic lowercase Latin letters \( \mc{g}, \ \mc{q}, \ \mc{r} \), and the straight lowercase Latin letters \( g, \ q, \ r \) denote the expectations of the gradients. 
A bar above the letter (i.e., \( \bar{g} \)) indicates that we take the average of this value among all devices.

We summarize these notation and other objects, which we use in the proofs for shortness in the following table.

\input{tables/Contractions}

%% file: tables/Contractions.tex
\renewcommand{\arraystretch}{2.5} 
\begin{table}[H]
\centering
\begin{tabular}{|c|p{10cm}|}
\hline
\textbf{Notation} & \textbf{Meaning} \\
\hline




\(\bar{x}_t\) & \( \fsumm{m=1}{M} x^m_t \) - average of the weights among all devices at iteration $t$ \\


\( \bar{F}_t, \ \bar{Q}_t, \ \bar{R}_t \) & \( \fsumm{m=1}{M} F(x^m_t),\ \fsumm{m=1}{M} Q(x^m_t),\ \fsumm{m=1}{M} R(x^m_t) \) -- average of function values \\

\( \mc{g}^m_t, \ \mc{q}^m_t, \ \mc{r}^m_t \) & \( \nabla F(x^m_t, z^m_t),\ \nabla Q(x^m_t, z^m_t),\ \nabla R(x^m_t, z^m_t) \) -- corresponding stochastic gradients at iteration $t$ on device $m$ \\

\( \bar{\mc{g}}_t, \ \bar{\mc{q}}_t, \ \bar{\mc{r}}_t \) & \( \fsumm{m=1}{M} \mc{g}^m_t,\ \fsumm{m=1}{M} \mc{q}^m_t,\ \fsumm{m=1}{M} \mc{r}^m_t \) -- average of stochastic gradients at iteration $t$ \\

\( g^m_t, \ q^m_t, \ r^m_t \) & \( \nabla F(x^m_t),\ \nabla Q(x^m_t),\ \nabla R(x^m_t) \)  -- expected value of stochastic gradients at iteration $t$ on device $m$ (namely, 
\( \E[\mc{g}^m_t], \ \E[\mc{q}^m_t], \ \E[\mc{r}^m_t] \)) \\

\( \bar{{g}}_t, \ \bar{{q}}_t, \ \bar{{r}}_t \) & \( \fsumm{m=1}{M} {g}^m_t,\ \fsumm{m=1}{M} {q}^m_t,\ \fsumm{m=1}{M} {r}^m_t \) -- average of expected values of gradients at iteration $t$ \\


\( V_t \) & \( \fsumm{m=1}{M} \norm{x^m_t-\bar{x}_t}^2 \) -- mean deviation of $x^m_t$ from $\bar{x}_t$\\

\( \norm{r_t} \) & \( \norm{\bar{x_t} - x_*} \) -- distance between average weights and optimum at iteration~$t$ \\

\( \kappa \) & \( \frac{L}{\mu} \) -- condition number \\

\hline

\end{tabular}
\caption{Notation summary}
\label{tab:notation}
\end{table}

%% file: stuff/supplementary/BasicFacts.tex
\begin{lemma}[Young's inequality] For all $u,v\in \mathbb{R}^d$ and any positive $w$ the Young's inequality holds:
\begin{align}
    2\left\langle u, v \right\rangle \leq  w \| u\|^2 + \frac{1}{w}\| v \|^2. \label{ap:Cauchy–Schwarz}
\end{align}
\end{lemma}

\begin{lemma}[Jensen's inequality] For any convex function $g: \mathbb{R}^d \to \mathbb{R}$, all $\{u_i \}_{i=1}^n \in \mathbb{R}^d$ and any non-negative $\{p_i\}_{i=1}^n \in \mathbb{R}$ such that $\sum_{i=1}^n p_i = 1$ the Jensen's inequality holds:
\begin{align}
    g\left( \sum_{i=1}^n p_i u_i\right) \leq \sum_{i=1}^n p_i g(u_i). \label{ap:Jensen}
\end{align}
\end{lemma}

%% file: stuff/supplementary/TechnicalLemmas.tex
Before demonstrating the proof of the claimed facts, let us first establish some technical results.

\begin{lemma} \label{lem:lem_1}
For the averaged variables $\bar x_t$ obtained by the iterates of Algorithm \ref{alg:localsgd} for the problem \eqref{problem_main} under Assumption \ref{st:st_1} it holds that
    \begin{align*}
        \norm{\bar{x}_t - x_* - \gamma_{t} \bar{g}_t}^2
        =& \norm{\bar{x}_t - x_*}^2 
        + \gamma_{t}^2 \norm{\nabla Q(\bar x_t) 
        + \bar{r}_t - \nabla Q(x_*) - \nabla R(x_*)}^2 
        \\
        &- 2 \gamma_{t} \inner{\bar{x}_t - x_*}{\nabla Q(\bar x_t)}
        - 2 \gamma_{t} \inner{\bar{x}_t - x_*}{\bar{r}_t}.
    \end{align*}
\end{lemma}
\begin{proof}
To prove this lemma it is enough to use simple algebraic transformations and the notation introduced in Table \ref{tab:notation} for the functions $Q$ and $R$ from Assumption \ref{st:st_1}.
    \begin{align*}
        \|\bar{x}_t - & x_*  - \gamma_{t} \bar{g}_t \|^2
        \\
        =& 
        \norm{\bar{x}_t - x_*}^2 
        + \gamma_{t}^2 \norm{\bar{g}_t}^2 
        - 2 \gamma_{t} \inner{\bar{x}_t - x_*}{\bar{g}_t} 
        \\
        =& 
        \norm{\bar{x}_t - x_*}^2 
        + \gamma_{t}^2 \norm{\fsumm {m=1}{M}{\nabla F(x^m_t)}}^2 
        - 2\gamma_{t} \inner{\bar{x}_t - x_*}{ \frac{1}{M}\summ {m=1}{M} \nabla F(x^m_t)} \\
        =& 
            \notag
        \norm{\bar{x}_t - x_*}^2 
        + \gamma_{t}^2 \norm{\fsumm {m=1}{M}{\nabla F(x^m_t)} - \nabla F(x_*)}^2
        \\
        &- 2\gamma_{t} \inner{\bar{x}_t - x_*}{\frac{1}{M}\summ {m=1}{M} \nabla F(x^m_t)}\\
        =& 
            \notag
        \norm{\bar{x}_t - x_*}^2 
        + \gamma_{t}^2 \norm{\fsumm {m=1}{M}{ \left[\nabla Q(x^m_t) + \nabla R(x^m_t) - \nabla Q(x_*) - \nabla R(x_*)\right]}}^2  \\
        &- 2\gamma_{t}\inner{\bar{x}_t - x_*}{ \frac{1}{M}\summ {m=1}{M}\nabla Q(x^m_t)}
        - 2\gamma_{t} \inner{\bar{x}_t - x_*}{ \frac{1}{M}\summ {m=1}{M} \nabla R(x^m_t)}  \\
        =&
        \notag
        \norm{\bar{x}_t - x_*}^2 
        + \gamma_{t}^2 \norm{\fsumm {m=1}{M}{ \left[q^m_t + r^m_t - \nabla Q(x_*) - \nabla R(x_*) \right]}}^2  \\
        &- 2\gamma_{t}\inner{\bar{x}_t - x_*}{ \frac{1}{M}\summ {m=1}{M} q^m_t}
        - 2\gamma_{t} \inner{\bar{x}_t - x_*}{ \frac{1}{M}\summ {m=1}{M} r^m_t}
        \\
        =& 
        \norm{\bar{x}_t - x_*}^2 
        + \gamma_{t}^2 \norm{\frac{1}{M}\summ{m=1}{M} {q}^m_t 
        + \bar{r}_t - \nabla Q(x_*) - \nabla R(x_*)}^2 
        \\
        &- 2 \gamma_{t} \inner{\bar{x}_t - x_*}{\frac{1}{M}\summ{m=1}{M}  {q}^m_t} - 2 \gamma_{t} \inner{\bar{x}_t - x_*}{\bar{r}_t}.
    \end{align*}
    It remains to take into account quadraticity of the function $Q$. In particular, we use that $\frac{1}{M}\summ{m=1}{M} {q}^m_t = \nabla Q(\bar x_t)$. \qed
\end{proof}
\noindent Subsequently, we proceed to separately work with the terms of Lemma \ref{lem:lem_1}.
\begin{lemma} \label{lem:g_t}
    For the iterates of Algorithm \ref{alg:localsgd} for the problem \eqref{problem_main} under Assumption \ref{st:st_1} it holds that for any positive $\zeta$
    \begin{align*}
        \|\nabla Q(\bar x_t) 
        + \bar{r}_t - & \nabla Q(x_*) - \nabla R(x_*) \|^2 
        \\
        \leq&\
        2 L_Q(1 + \zeta)  (Q(\bar{x}_t) - Q(x_*) - \inner {\nabla Q(x_*)}{\bar{x}_t - x_*})
        \\
        &+ 2 L_R \left(1 + \frac{1}{\zeta} \right) (\bar{R}_t - R(x_*) - \inner {\nabla R(x_*)}{\bar{x}_t - x_*}).
    \end{align*}
\end{lemma}
\begin{proof} By the Young's inequality \eqref{ap:Cauchy–Schwarz} for any positive $\zeta$, we have
    \begin{align*}
        \|\nabla Q(\bar x_t) 
        + \bar{r}_t & - \nabla Q(x_*) - \nabla R(x_*) \|^2
        \\
        =&\ \norm{\nabla Q(\bar x_t) - \nabla Q(x_*)}^2 
        + \norm{\bar{r}_t - \nabla R(x_*)}^2
        \\
        &+ 2\inner {\nabla Q(\bar x_t) - \nabla Q(x_*)} {\bar{r}_t - \nabla R(x_*)} 
        \notag \\
        \leq&\ \norm{\nabla Q(\bar x_t) - \nabla Q(x_*)}^2 
        + \norm{\bar{r}_t - \nabla R(x_*)}^2
        \\
        &+ \zeta \norm{\nabla Q(\bar x_t) - \nabla Q(x_*)}^2 
        + \frac{1}{\zeta} \norm{\bar{r}_t - \nabla R(x_*)}^2 \notag \\
        =&\ (1 + \zeta)\norm{\nabla Q(\bar x_t) - \nabla Q(x_*)}^2 + \left(1 + \frac{1}{\zeta} \right)\norm{\bar{r}_t - \nabla R(x_*)}^2
        \notag \\
        =&\ (1 + \zeta)\norm{\nabla Q(\bar x_t) - \nabla Q(x_*)}^2 
        \\
        &+\left(1 + \frac{1}{\zeta} \right)\norm{\frac{1}{M} \sum_{m=1}^M [\nabla R(x^m_t) - \nabla R(x_*)]}^2
        \\
        \leq&\  (1 + \zeta)\norm{\nabla Q(\bar x_t) - \nabla Q(x_*)}^2 
        \\
        &+\left(1 + \frac{1}{\zeta} \right) \cdot \frac{1}{M} \sum_{m=1}^M \norm{ \nabla R(x^m_t) - \nabla R(x_*)}^2.
    \end{align*}
    Here we also used the notation for $\bar r_t$ and convexity of the norm \eqref{ap:Jensen}. With smoothness of $Q$ and $R$ (Assumption \ref{st:st_1} and Corollary~\ref{cor:nesterov}), one can obtain
    \begin{align*}
        \|\nabla Q(\bar x_t) 
        + \bar{r}_t & - \nabla Q(x_*) - \nabla R(x_*) \|^2
        \\
        \leq&\ 2 (1 + \zeta) L_Q ( Q(\bar x_t) - Q (x_*) - \inner {\nabla Q(x_*)}{\bar{x}_t - x_*})
        \\
        &+ 2\left(1 + \frac{1}{\zeta}\right) L_R \cdot \frac{1}{M} \sum_{m=1}^M [R (x^m_t) - R (x_*) - \inner {\nabla R(x_*)}{x^m_t - x_*}].
    \end{align*}
To finish the proof one just need to use the notation of $\bar R_t$. \qed
\end{proof}

\begin{lemma} \label{lem:inner_1}
For the iterates of Algorithm \ref{alg:localsgd} for the problem \eqref{problem_main} under Assumption \ref{st:st_1} it holds that
    \begin{align*}
        -2 \inner{\bar{x}_t - x_*}{\nabla Q(\bar x_t)} &\leq 2 Q (x_*) - 2 Q(\bar{x}_t) - \mu_Q \norm{\bar{x}_t - x_*}^2. 
    \end{align*}
\end{lemma}
\begin{proof}
    It is enough to apply the $\mu_Q$-convexity of $Q$. \qed
\end{proof}

\begin{lemma} \label{lem:inner_2}
For the iterates of Algorithm \ref{alg:localsgd} for the problem \eqref{problem_main} under Assumption \ref{st:st_1} it holds that 
\begin{align*}
    - 2 \inner{\bar{x}_t - x_*}{\bar{r}_t} 
    \leq& 
    - (\bar{R}_t - R (x_*))  + 2 L_R V_t - \mu_R \norm{x_* - \bar{x}_t}^2 
    \\
    &- \inner{\nabla R(x_*)}{\bar{x}_t - x_*}.
\end{align*}
\end{lemma}
\begin{proof}
We start from the notation of $r^m_t$ and get
    \begin{align}
        2 \inner{x_* - \bar{x}_t} {\bar{r}_t}
        &= \frac{2}{M} \summ{m=1}{M}{\inner{x_* - x^m_t + x^m_t - \bar{x}_t}{r^m_t}} 
        \notag\\
        &= \frac{2}{M} \summ{m=1}{M}{\inner{x_* - x^m_t}{r^m_t}} 
        + \frac{2}{M} \summ{m=1}{M}{\inner{x^m_t - \bar{x}_t}{r^m_t}}. 
        \label{eq:St_3_0}
    \end{align}
    We estimate the first term by the $\mu_R$ - strong convexity and the Jensen's inequality \eqref{ap:Jensen}:
    \begin{align}
        \frac{2}{M} \summ{m=1}{M}{\inner{x_* - x^m_t}{r^m_t}} 
        \leq&\ \frac{1}{M} \summ{m=1}{M}{2 R (x_*) - 2 R(x^m_t) - \mu_R \norm{x_* - x^m_t}^2} 
        \notag\\
        \leq&\ \frac{1}{M} \summ{m=1}{M}{2 R (x_*) - 2 R(x^m_t) - \mu_R \norm{x_* - \bar{x}_t}^2}
        \notag\\
        =&\ 2 R (x_*) - 2 \bar{R}_t - \mu_R \norm{x_* - \bar{x}_t}^2. 
        \label{eq:St_3_1}
    \end{align}
    For the second part we apply the Young's inequality \eqref{ap:Cauchy–Schwarz}:
    \begin{align}
        \frac{2}{M} \summ{m=1}{M}{\inner{x^m_t - \bar{x}_t}{r^m_t}}
        =&\ \frac{2}{M} \summ{m=1}{M}{\inner{x^m_t - \bar{x}_t}{r^m_t - \nabla R( x_*)}}
        \notag\\
        &+ \frac{2}{M} \summ{m=1}{M}{\inner{x^m_t - \bar{x}_t}{\nabla R( x_*)}} 
        \notag\\
        =&\ 
        \frac{2}{M} \summ{m=1}{M}{\inner{x^m_t - \bar{x}_t}{r^m_t - \nabla R( x_*)}}
        \notag\\
        \leq&\ \frac{2}{M} \summ{m=1}{M}{L_R \norm{x^m_t - \bar{x}_t}^2 }
        + \frac{1}{M} \summ{m=1}{M}{\frac{1}{2 L_R} \norm{r^m_t-\nabla R( x_*)}^2 } 
        \notag\\
        =&\ 
        2 L_R V_t 
        + \frac{1}{M} \summ{m=1}{M}{\frac{1}{2 L_R} \norm{r^m_t-\nabla R( x_*)}^2 }.
        \label{eq:St_3_2_1}
    \end{align} 
    Here we also used the notation of $V_t$. By Assumption \ref{st:st_1} on the smoothness of $R$:
    \begin{align}
        \norm{r^m_t- \nabla R(x_*)}^2 \leq 2L_R (R(x^m_t) - R (x_*) - \inner{\nabla R(x_*)}{x^m_t - x_*}).
        \label{eq:St_3_2_2}
    \end{align}
    Substituting \eqref{eq:St_3_2_2} into \eqref{eq:St_3_2_1}, we complete the second part:
    \begin{align}
        \frac{2}{M} \summ{m=1}{M}{\inner{x^m_t - \bar{x}_t}{r^m_t}}
        &\leq 2 L_R V_t 
        + \frac{1}{M} \summ{m=1}{M}{ \left[R(x^m_t) - R (x_*) - \inner {\nabla R(x_*)}{x^m_t - x_*} \right]} \notag\\
        &\leq 2 L_R V_t 
        + (\bar{R}_t - R (x_*) - \inner{\nabla R(x_*)}{\bar{x}_t - x_*}). \label{eq:St_3_2}
    \end{align}
    Combining together \eqref{eq:St_3_0}, \eqref{eq:St_3_1}, \eqref{eq:St_3_2}, we gain:
    \begin{align*}
         - 2 \inner{\bar{x}_t - x_*}{\bar{r}_t} 
         \leq&\
         2 R (x_*) - 2 \bar{R}_t - \mu_R \norm{x_* - \bar{x}_t}^2 
         + 2 L_R V_t 
         \\
         &+ \left(\bar{R}_t - R (x_*) - \inner {\nabla R (x_*)}{\bar{x}_t - x_*}\right) 
         \\ 
         =& - \left(\bar{R}_t - R (x_*)\right)  + 2 L_R V_t - \mu_R \norm{x_* - \bar{x}_t}^2
         \\
         &- \inner{\nabla R (x_*)}{\bar{x}_t - x_*}.
    \end{align*}
    This finishes the proof. \qed
\end{proof}

\begin{lemma}\label{lem:AB-lemma}
For the iterates of Algorithm \ref{alg:localsgd} with $\gamma_t \leq \frac{1}{6L}$ for the problem \eqref{problem_main} under Assumption \ref{st:st_1} it holds that
    \begin{align*}
        \norm{\bar{x}_t - x_* - \gamma_{t} \bar{g}_t}^2 
        \leq&
        (1 - \gamma_{t} \mu)\norm{\bar{x}_t - x_*}^2 - \frac{\gamma_{t}}{6} \left(F(\bar{x}_t) - F (x_*)\right) + 2 \gamma_{t} L_R V_t.
\end{align*}
\end{lemma}
\begin{proof} Substituting the results of Lemmas~\ref{lem:g_t},~\ref{lem:inner_1} and~\ref{lem:inner_2} into Lemma~\ref{lem:lem_1} and making simple algebra:
    \begin{align*}
            \norm{\bar{x}_t - x_* - \gamma_{t} \bar{g}_t}^2 
            \leq& \norm{\bar{x}_t - x_*}^2 
            \\
            &+ 2\gamma_{t}^2 L_Q (1 + \zeta) \left(Q(\bar{x}_t) - Q (x_*) - \inner{\nabla Q(x_*)}{\bar{x}_t - x_*}\right) 
            \\
            &+ 2\gamma_{t}^2 L_R \left(1 + \frac{1}{\zeta} \right) \left(\bar{R}_t - R (x_*) - \inner {\nabla R(x_*)}{\bar{x}_t - x_*}\right) 
            \\
            &+ \gamma_{t} \left( 2 Q (x_*) - 2 Q(\bar{x}_t) - \mu_Q \norm{\bar{x}_t - x_*}^2 \right) 
            \\
            &-\gamma_{t} \left(\bar{R}_t - R (x_*) + 2  L_R V_t - \mu_R \norm{x_* - \bar{x}_t}^2\right) 
            \\
            &- \gamma_{t}\inner{\nabla R(x_*)}{\bar{x}_t - x_*} 
            \\
            =& (1 - \gamma_{t} \mu_Q - \gamma_{t} \mu_R)\norm{\bar{x}_t - x_*}^2 + 2 \gamma_{t} L_R V_t 
            \\
            &+ 2 \gamma_t \coef {\gamma_{t} L_Q (1+\zeta) - 1} (Q(\bar{x}_t) - Q (x_*))
            \\
            &+ 2\gamma_{t} \coef{\gamma_{t} L_R \left(1+\frac{1}{\zeta}\right) + \frac{1}{2} - 1 } (\bar{R}_t - R (x_*))
            \\
            &- 2 \gamma_t \coef{\gamma_{t} L_Q(1+\zeta)} \inner{ \nabla Q(x_*)}{\bar{x}_t - x_*}
            \\
            &- 2 \gamma_t \coef{\gamma_{t} L_R \left(1+\frac{1}{\zeta} \right) + \frac{1}{2}} \inner{\nabla R(x_*)}{\bar{x}_t - x_*}.  
    \end{align*}
    We want to achieve $\gamma_{t} L_R \left(1+\frac{1}{\zeta} \right) + \frac{1}{2} = \gamma_{t} L_Q(1+\zeta)$. One can note that it is equivalent to find a positive root of 
    $$
    L_Q \zeta^2  + \left(L_Q - L_R - \frac{1}{2\gamma_{t}}\right) \zeta - L_R = 0.
    $$
    In particular, we take 
    $$
    \zeta = \frac{- \left(L_Q - L_R - \frac{1}{2\gamma_{t}} \right) + \sqrt{ \left(L_Q - L_R - \frac{1}{2\gamma_{t}}\right)^2 + 4 L_Q L_R}}{2L_Q} > 0.
    $$
    With such $\zeta$, we have
    \begin{align*}
            \norm{\bar{x}_t - x_* - \gamma_{t} \bar{g}_t}^2 
            \leq& 
            (1 - \gamma_{t} \mu_Q - \gamma_{t} \mu_R)\norm{\bar{x}_t - x_*}^2 + 2 \gamma_{t} L_R V_t 
            \\
            &+ 2 \gamma_t \coef {\gamma_{t} L_Q (1+\zeta) - 1} (Q(\bar{x}_t) - Q (x_*))
            \\
            &+ 2\gamma_{t} \coef {\gamma_{t} L_Q (1+\zeta) - 1} (\bar{R}_t - R (x_*))
            \\
            &- 2 \gamma_t \coef{\gamma_{t} L_Q(1+\zeta)} \inner{ \nabla R(x_*) + \nabla Q(x_*)}{\bar{x}_t - x_*}.  
    \end{align*}
    Taking into account the optimality condition $\nabla R(x_*) + \nabla Q(x_*) = 0$, we get
    \begin{align}
            \norm{\bar{x}_t - x_* - \gamma_{t} \bar{g}_t}^2 
            \leq& 
            (1 - \gamma_{t} \mu_Q - \gamma_{t} \mu_R)\norm{\bar{x}_t - x_*}^2 + 2 \gamma_{t} L_R V_t 
            \notag\\
            &+ 2 \gamma_t \coef {\gamma_{t} L_Q (1+\zeta) - 1} (Q(\bar{x}_t) - Q (x_*))
            \notag\\
            &+ 2\gamma_{t} \coef {\gamma_{t} L_Q (1+\zeta) - 1} (\bar{R}_t - R (x_*)).
            \label{eq:temp2}
    \end{align}
    Next, we prove that $\gamma_{t} L_Q(1+\zeta) \leq 1$ for $\gamma_t \leq \frac{1}{6L}$ with $L = L_Q + L_R$:
    \begin{align*}
        \gamma_{t} L_Q(1+\zeta) 
        &= \gamma_{t} L_Q 
        \coef {1 + \frac{- \left(L_Q - L_R - \frac{1}{2\gamma_{t}}\right)+\sqrt{\left(L_Q - L_R -\frac{1}{2\gamma_{t}}\right)^2 + 4 L_Q L_R}}{2L_Q} }
        \notag\\
        &=
        \frac{\gamma_{t}}{2} 
        \coef {2L_Q - \left(L_Q - L_R - \frac{1}{2\gamma_{t}}\right) + \sqrt{ \left(L_Q - L_R - \frac{1}{2\gamma_{t}} \right)^2 + 4 L_Q L_R} } 
        \notag\\ 
        & \leq
        \frac{\gamma_{t}}{2} 
        \coef { L_Q + L_R + \frac{1}{2\gamma_{t}} + \left|L_Q - L_R - \frac{1}{2\gamma_{t}} \right| + \sqrt{ 4 L_Q L_R} }
        \notag\\
        & \leq
        \frac{\gamma_{t}}{2} 
        \coef {L_Q + L_R + \frac{1}{2\gamma_{t}} + L + \frac{1}{2\gamma_{t}} + \sqrt{ 4 L^2} }
        \notag\\
        & \leq
        \frac{\gamma_{t}}{2} 
        \coef {5 L + \frac{1}{\gamma_{t}} }
        \leq
        \frac{\gamma_{t}}{2} 
        \coef {\frac{5}{6 \gamma_{t}} + \frac{1}{\gamma_{t}} } 
        \notag\\
        &=
        \frac{5}{12} + \frac{1}{2} \leq \frac{11}{12}.
    \end{align*}
    This fact gives that $\coef {\gamma_{t} L_Q (1+\zeta) - 1} \leq -\frac{1}{12}$, therefore, one can use the Jensen's inequality \eqref{ap:Jensen} to deduce:
    \begin{align}
        \label{eq:temp1}
        \coef {\gamma_{t} L_Q (1+\zeta) - 1} \bar{R}_t 
        =& \coef {\gamma_{t} L_Q (1+\zeta) - 1} \cdot \frac{1}{M} \sum_{m=1}^M R(x^m_t) 
        \notag\\
        \leq& \coef {\gamma_{t} L_Q (1+\zeta) - 1} \cdot R \left( \frac{1}{M} \sum_{m=1}^M x^m_t \right)
        \notag\\
        =& \coef {\gamma_{t} L_Q (1+\zeta) - 1} R (\bar x_t).
    \end{align}
    Combining \eqref{eq:temp2} and \eqref{eq:temp1}, we get
    \begin{align*}
            \norm{\bar{x}_t - x_* - \gamma_{t} \bar{g}_t}^2 
            \leq& 
            (1 - \gamma_{t} \mu_Q - \gamma_{t} \mu_R)\norm{\bar{x}_t - x_*}^2 + 2 \gamma_{t} L_R V_t 
            \notag\\
            &+ 2 \gamma_t \coef {\gamma_{t} L_Q (1+\zeta) - 1} (Q(\bar{x}_t) - Q (x_*))
            \notag\\
            &+ 2\gamma_{t} \coef {\gamma_{t} L_Q (1+\zeta) - 1} (R( \bar x_t) - R (x_*))
            \\
            =& 
            (1 - \gamma_{t} \mu_Q - \gamma_{t} \mu_R)\norm{\bar{x}_t - x_*}^2 + 2 \gamma_{t} L_R V_t 
            \notag\\
            &+ 2 \gamma_t \coef {\gamma_{t} L_Q (1+\zeta) - 1} (Q(\bar{x}_t) + R( \bar x_t) - Q (x_*) - R (x_*))
            \\
            \leq& 
            (1 - \gamma_{t} \mu_Q - \gamma_{t} \mu_R)\norm{\bar{x}_t - x_*}^2 + 2 \gamma_{t} L_R V_t
            \\
            &-\frac{\gamma_t}{6} (F(\bar x_t) - F(x_*)).
    \end{align*}
    Utilizing that $\mu_Q + \mu_R = \mu$, we come to the statement of the lemma. \qed
\end{proof}
Now we are ready to combine the obtained results and come to the descent lemma.
\begin{lemma} \label{lem:very_main}
    For the averaged variables $\bar x_t$ obtained by the iterates of Algorithm \ref{alg:localsgd} with $\gamma_t \leq \frac{1}{6L}$ for the problem \eqref{problem_main} under Assumption \ref{st:st_1} it holds that
    \begin{align*}
        \E \norm{\bar{x}_{t+1}-x_*}^2
        \leq& (1 - \gamma_{t} \mu) \E \norm{\bar{x}_t - x_*}^2 
        + \gamma_{t}^2 \E \norm{ \bar{\mc{g_t}} - \bar{g}_t}^2
        \\
        &- \frac{\gamma_{t}}{6} \E [F(\bar{x}_t) - F_*] 
        + 2 \gamma_{t} L_R \E [V_t]. 
    \end{align*}
\end{lemma}
\begin{proof}
    Using the update of Algorithm \ref{alg:localsgd}, we have
    \begin{align*}
        \norm{\bar{x}_{t+1}-x_*}^2
        &= \norm{\bar{x}_{t} - \gamma_{t} \bar {\mc{g_t}} - x_*}^2
        = \norm{\bar{x}_{t} - \gamma_{t} \bar {\mc{g_t}} - x_* - \gamma_{t} \bar{g}_t + \gamma_{t} \bar{g}_t}^2 
        \\
        &= \norm{\bar{x}_t - x_* - \gamma_{t} \bar{g}_t}^2 
        + \gamma_{t}^2 \norm{\bar {\mc{g_t}} - \bar{g}_t}^2
        - 2 \gamma_{t} \inner{\bar{x}_t - x_* - \gamma_{t} \bar{g}_t}{\bar {\mc{g_t}} - \bar{g}_t}.
    \end{align*}
    Taking the expectation, 
    \begin{align*}
        \E \norm{\bar{x}_{t+1}-x_*}^2
        &= \E \norm{\bar{x}_t - x_* - \gamma_{t} \bar{g}_t}^2
        + \gamma_{t}^2 \E \norm{ \bar {\mc{g_t}} - \bar{g}_t}^2.
    \end{align*}
With the results of Lemma~\ref{lem:AB-lemma}:
    \begin{align*}
         \E \norm{\bar{x}_{t+1}-x_*}^2
        \leq&\ (1 - \gamma_{t} \mu) \E \norm{\bar{x}_t - x_*}^2 
        + \gamma_{t}^2 \E \norm{ \bar{\mc{g_t}} - \bar{g}_t}^2
        \\
        &- \frac{\gamma_{t}}{6} \E [F(\bar{x}_t) - F_*] 
        + 2 \gamma_{t} L_R \E [V_t].
    \end{align*}
This finishes the proof. \qed
\end{proof}

%% file: stuff/supplementary/OtherLemmas.tex

In the previous section, we proved the descent lemma. It remains to estimate $\E \norm{ \bar{\mc{g_t}} - \bar{g}_t}^2$ and $\E [V_t]$. 
\begin{lemma} \label{lem:rho_2}
For the stochastic gradients from Algorithm \ref{alg:localsgd} for the problem \eqref{problem_main} under Assumptions \ref{st:st_1} and \ref{ass:ass_2} it holds that
    \begin{align*}
        \E \norm{\bar{\mc{g}}_t - \bar{g}_t}^2 
        \leq 
        \frac{\sigma^2}{M} + \frac{\rho L^2 }{M} V_t  + \frac{\rho L^2 }{M} \norm{\bar{x}_t - x_*}^2. 
    \end{align*}
\end{lemma}
\begin{proof}
Firstly, we note that all $\mc{g}^m_t$ are independent and unbiased, it gives that
    \begin{align}
        \label{eq:temp4}
        \E \norm{\bar{\mc{g}}_t - \bar{g}_t}^2 
        &=
        \E \norm{\fsumm{m=1}{M}{[\mc{g}^m_t - g^m_t]}}^2 = \frac{1}{M^2} \summ{m=1}{M} \E \norm{\mc{g}^m_t - g^m_t}^2. 
    \end{align}
Next, we use Assumption \ref{ass:ass_2} to get
    \begin{align}
        \label{eq:temp3}
        \frac{1}{M} \summ{m=1}{M} \E \norm{\mc{g}^m_t - g^m_t}^2
        \leq&\
        \sigma^2 + \frac{\rho}{M} \summ{m=1}{M} \norm {g^m_t}^2 
        \nonumber \\
        =&\
        \sigma^2 + \frac{\rho}{M} \summ{m=1}{M} \norm {\nabla F(x^m_t) - \nabla F(x_*)}^2 
        \notag \\
        \leq&\ 
        \sigma^2 + \frac{\rho L^2}{M} \summ{m=1}{M} \norm{x^m_t - x_*}^2 
        \notag \\
        =&\
        \sigma^2 + \frac{\rho L^2}{M} \summ{m=1}{M} \norm{x^m_t - \bar{x}_t + \bar{x}_t - x_*}^2 
        \notag \\
        =&\
        \sigma^2 + \frac{\rho L^2}{M} \summ{m=1}{M} \left( \norm{x^m_t - \bar{x}_t}^2
        + \norm{\bar{x}_t - x_*}^2 \right)
        \notag \\
        &+ \frac{2\rho L^2}{M} \summ{m=1}{M} \inner{x^m_t - \bar{x}_t}{\bar{x}_t - x_*} 
        \notag \\
        =&\
        \sigma^2 + \frac{\rho L^2}{M} \summ{m=1}{M} \left( \norm{x^m_t - \bar{x}_t}^2
        + \norm{\bar{x}_t - x_*}^2 \right)
        \notag \\
        =&\
        \sigma^2 + \rho L^2 V_t
        + \rho L^2 \norm{\bar{x}_t - x_*}^2.
    \end{align}
Substituting \eqref{eq:temp3} into \eqref{eq:temp4} and applying the definition of $V_t$, we get the final result. \qed
\end{proof}

\begin{lemma} \label{lem:lemma_Vt}
    Suppose that Assumptions \ref{st:st_1} and \ref{ass:ass_2} hold.
    \begin{itemize}
        \item[(a)] If $\rho = 0$ and $\gamma_{t} = \gamma \leq \frac{1}{6L}$, then
        \begin{align}
            \label{eq: bound-1}
            \E [V_t] \leq (H-1) \sigma^2 \gamma^2; 
        \end{align}
        \item[(b)] If $\rho = 0$, $\mu > 0$ and $\gamma_t = \frac{2}{\mu(\xi + t + 1)}$ with $\xi \geq 0$, then 
            \begin{align}
                \label{bound-2}
                \E [V_t] \leq \frac{2 (M - 1) (H - 1) \sigma^2 \gamma_{t-1}^2}{M};
            \end{align}
        \item[(c)] If $\rho = 0$ and $\gamma_t$ is non-decreasing, then
            \begin{align}
                \label{bound-3}
                \E [V_t] \leq \frac{2 (M-1) (H-1) \sigma^2 \gamma_{t-H+1 \land 0}}{M};
            \end{align}
        \item[(d)] If $\mu > 0$ and $\gamma_t \leq \min\left\{ \frac{\mu}{3\rho L^2}, \frac{1}{6L} \right\}$, then 
            \begin{align}
                \label{bound-4}
                \E [V_t] \leq
                \summ{i=kH}{kH+a - 1} \left(
                \rho \gamma_{i}^2 L^2 \E \norm{r_i}^2 + \gamma_{i}^2 \sigma^2\right)
                \prod\limits_{j = i + 1}^{kH+a - 1}\left(1 - \frac{\gamma_{j} \mu}{2}\right)
            \end{align}
            with $t = kH + a$.
    \end{itemize}




\end{lemma}

\begin{proof}
    First, note that cases $(a)$, $(b)$ and $(c)$ have been already proved ($(a)$ in \cite{khaled2020tighter} and $(b)$, $(c)$ in \cite{woodworth2020local}). Next, let us prove $(d)$.









\noindent For $t \in \mathbb{N}$ we have $x^m_{t+1} = x^m_t - \gamma_{t} \mc{g}^m_t$ 
and $\bar{x}_{t+1} = \bar{x}_t - \gamma_{t} \mc{\bar{g}}_t$ if $(t + 1) \mod H \neq 0$. Hence, for such $t$ and for conditioned expectation it is true that:
    \begin{align*}
        \E \norm{x^m_{t+1} - \bar{x}_{t+1}}^2 
        =&\ \norm{x^m_t - \bar{x}_t}^2 + \gamma_{t}^2 \E \norm{\mc{g}^m_t - \mc{\bar{g}}_t}^2 - 2\gamma_{t} \E [\inner{x^m_t - \bar{x}_t}{\mc{g}^m_t - \mc{\bar{g}}_t}] \\
        =&\ \norm{x^m_t - \bar{x}_t}^2 + \gamma_{t}^2 \E \norm{\mc{g}^m_t - \mc{\bar{g}}_t}^2 - 2\gamma_{t} \inner{x^m_t - \bar{x}_t}{g^m_t} \\ &+ 2\gamma_{t} \inner {x^m_t - \bar{x}_t}{\bar{g}_t}.
    \end{align*}
Averaging over $m$, the last terms collapse to zero. Therefore,
    \begin{align}
        \E [V_{t+1}] =
        V_t + \frac{\gamma_{t}^2}{M} \summ{m=1}{M} \E \norm{\mc{g}^m_t - \mc{\bar{g}}_t}^2 - \frac{2\gamma_{t}}{M} \summ{m=1}{M} \inner{x^m_t - \bar{x}_t}{g^m_t}. \label{eq:1904_4}
    \end{align}
By expanding square into the second term of \eqref{eq:1904_4}, we get
    \begin{align} 
        \E \norm{\mc{g}^m_t - \mc{\bar{g}}_t}^2
        &= \E \norm{\mc{g}^m_t - \bar{g}_t + \bar{g}_t - \mc{\bar{g}}_t}^2 \nonumber \\
        &= 
        \E \norm{\mc{g}^m_t - \bar{g}_t}^2 + \E \norm{\mc{\bar{g}}_t - \bar{g}_t}^2 + 2\E [\inner{\mc{g}^m_t - \bar{g}_t}{\bar{g}_t - \mc{\bar{g}}_t}]. \label{eq:1904_1}
    \end{align}
And again:
    \begin{align}
        \E \norm{\mc{g}^m_t - \bar{g}_t}^2 
        &= \E \norm{\mc{g}^m_t - g^m_t + g^m_t - \bar{g}_t}^2 \nonumber\\
        &= \E \norm{\mc{g}^m_t - g^m_t}^2 
        + \norm{g^m_t - \bar{g}_t}^2
        + 2\E[\inner{\mc{g}^m_t - g^m_t}{g^m_t - \bar{g}_t}] \nonumber \\
        &= \E \norm{\mc{g}^m_t - g^m_t}^2 
        + \norm{g^m_t - \bar{g}_t}^2
        + 2 \inner{g^m_t - g^m_t}{g^m_t - \bar{g}_t} \nonumber \\
        &= \E \norm{\mc{g}^m_t - g^m_t}^2 
        + \norm{g^m_t - \bar{g}_t}^2. \label{eq:1904_2}
    \end{align}
Combining \eqref{eq:1904_1} and \eqref{eq:1904_2}, we have
    \begin{align*}
        \E \norm{\mc{g}^m_t - \mc{\bar{g}}_t}^2
        =&\ \E \norm{\mc{g}^m_t - g^m_t}^2 
        + \norm{g^m_t - \bar{g}_t}^2 + \E \norm{\mc{\bar{g}}_t - \bar{g}_t}^2 \nonumber\\&+ 2\E [\inner{\mc{g}^m_t - \bar{g}_t}{\bar{g}_t - \mc{\bar{g}}_t}].
    \end{align*}
By averaging both sides over $m$:
    \begin{align}
        \fsumm{m=1}{M} \E \norm{\mc{g}^m_t - \mc{\bar{g}}_t}^2
        =&\ \fsumm{m=1}{M} \E \norm{\mc{g}^m_t - g^m_t}^2 
        + \fsumm{m=1}{M} \norm{g^m_t - \bar{g}_t}^2 \nonumber\\
        &+ \E \norm{\mc{\bar{g}}_t - \bar{g}_t}^2 
        + 2\E [\inner{\mc{\bar{g}}_t - \bar{g}_t}{\bar{g}_t - \mc{\bar{g}}_t}] \nonumber\\
        =&\ \fsumm{m=1}{M} \E \norm{\mc{g}^m_t - g^m_t}^2 
        + \fsumm{m=1}{M} \norm{g^m_t - \bar{g}_t}^2 \nonumber\\
        &+ \E \norm{\mc{\bar{g}}_t - \bar{g}_t}^2 
        - 2\E \norm{\mc{\bar{g}}_t - \bar{g}_t}^2 \nonumber\\
        \leq&\ \fsumm{m=1}{M} \E \norm{\mc{g}^m_t - g^m_t}^2 
        + \fsumm{m=1}{M} \norm{g^m_t - \bar{g}_t}^2. \label{eq:1904_3}
    \end{align}
Using \cref{st:st_1}, we bound the second term here as follows:
    \begin{align}
        \fsumm{m=1}{M} \norm{g^m_t - \bar{g}_t}^2 
        =&\
        \fsumm{m=1}{M} \norm{g^m_t - \nabla F(\bar{x}_t) + \nabla F(\bar{x}_t) - \bar{g}_t}^2 \nonumber\\
        =&\
        \fsumm{m=1}{M} \norm{g^m_t - \nabla F(\bar{x}_t)}^2 + \norm{\nabla F(\bar{x}_t) - \bar{g}_t}^2 \nonumber \\&+ \frac{2}{M}\sum\limits_{m=1}^M\inner{g^m_t - \nabla F(\bar{x}_t)}{\nabla F(\bar{x}_t) - \bar{g}_t} \nonumber \\
        =&\ 
        \fsumm{m=1}{M} \norm{g^m_t - \nabla F(\bar{x}_t)}^2 + \norm{\nabla F(\bar{x}_t) - \bar{g}_t}^2 \nonumber \\ &- 2 \norm{\nabla F(\bar{x}_t) - \bar{g}_t}^2 \nonumber\\
        =&\ 
        \fsumm{m=1}{M} \norm{g^m_t - \nabla F(\bar{x}_t)}^2 - \norm{\nabla F(\bar{x}_t) - \bar{g}_t}^2 \nonumber \\
        \leq&\ 
        \fsumm{m=1}{M} \norm{g^m_t - \nabla F(\bar{x}_t)}^2 \nonumber \\
        =&\
        \fsumm{m=1}{M} \norm{\nabla F(x^m_t) - \nabla F(\bar{x}_t)}^2 \nonumber\\
        {\leq}&\
        \fsumm{m=1}{M} 2L \left(F(\bar{x}_t) - F(x^m_t) - \inner{\bar{x}_t - x^m_t}{\nabla F(x^m_t)}\right) \nonumber\\
        \overset{\eqref{ap:Jensen}}{\leq}&\
        \frac{2L}{M} \summ{m=1}{M} \inner{x^m_t - \bar{x}_t}{\nabla F(x^m_t)}.
        \label{eq:2004_1}
    \end{align}
Substituting \eqref{eq:2004_1} into \eqref{eq:1904_3} and applying \cref{lem:rho_2}, one can obtain
    \begin{align*}
        \fsumm{m=1}{M} \E \norm{\mc{g}^m_t - \mc{\bar{g}}_t} 
        \leq&\
        \fsumm{m=1}{M} \E \norm{\mc{g}^m_t - g^m_t}^2 
        +
        \frac{2L}{M} \summ{m=1}{M} \inner{x^m_t - \bar{x}_t}{\nabla F(x^m_t)} \nonumber\\
        \leq&\
        \sigma^2 + \rho L^2 V_t + \rho L^2 \norm{r_t}^2
        \nonumber\\ &+
        \frac{2L}{M} \summ{m=1}{M} \inner{x^m_t - \bar{x}_t}{\nabla F(x^m_t)}.
    \end{align*}
Consequently, combining the result above with \eqref{eq:1904_4}, we have
    \begin{align}
        \E [V_{t+1}] 
        =&\
        V_t + \frac{\gamma_{t}^2}{M} \summ{m=1}{M} \E \norm{\mc{g}^m_t - \mc{\bar{g}}_t}^2 - \frac{2\gamma_{t}}{M} \summ{m=1}{M} \inner{x^m_t - \bar{x}_t}{\nabla F(x^m_t)}\nonumber \\
        \leq&\ 
        V_t 
        + \gamma_{t}^2 \sigma^2 + \gamma_{t}^2 \rho L^2 V_t + \gamma_{t}^2 \rho L^2 \norm{r_t}^2
        \nonumber \\&-\ \frac{2\gamma_{t}}{M}(1 - \gamma_{t} L) \summ{m=1}{M} \inner{x^m_t - \bar{x}_t}{\nabla F(x^m_t)}.
        \label{eq:1904_5}
    \end{align}
Now let us analyize last term. We know that $\gamma_{t} \leq \frac{1}{6L}$, therefore $1 - \gamma_{t} L \geq 0$. Thus, by the strong convexity with $\mu$ and the Jensen's inequality \eqref{ap:Jensen}:
    \begin{align}
        - \frac{2\gamma_{t}}{M}(1 - \gamma_{t} L) &\sum\limits_{m=1}^{M}\inner{x^m_t - \bar{x}_t}{\nabla F(x^m_t)} 
        =
        \frac{2\gamma_{t}}{M}(1 - \gamma_{t} L) \summ{m=1}{M} \inner{\bar{x}_t - x^m_t}{\nabla F(x^m_t)} \nonumber\\
        &{\leq}\
        \frac{2\gamma_{t}}{M}(1 - \gamma_{t} L) \summ{m=1}{M} \left(F(\bar{x}_t) - F(x^m_t) - \frac{\mu}{2} \norm{x^m_t - \bar{x}_t}^2 \right) \nonumber\\
        &\overset{\eqref{ap:Jensen}}{\leq}
        - \frac{\gamma_{t}}{M}(1 - \gamma_{t} L) \summ{m=1}{M} \mu \norm{x^m_t - \bar{x}_t}^2 
        =
        - \gamma_{t} (1 - \gamma_{t} L) \mu V_t.
        \label{eq:1904_6}
    \end{align}
Plugging \eqref{eq:1904_6} into \eqref{eq:1904_5} and using the boundary $\gamma_{t} \leq \frac{1}{6L}$, we obtain
    \begin{align}
    \label{eq: recur-V}
        \E [V_{t+1}]
        &\leq
        V_t 
        + \gamma_{t}^2 \sigma^2 + \gamma_{t}^2 \rho L^2 V_t + \gamma_{t}^2 \rho L^2 \norm{r_t}^2
        - \gamma_{t} (1 - \gamma_{t} L) \mu V_t \nonumber\\
        &=
        (1 - \gamma_{t} (1 - \gamma_{t} L) \mu) V_t
        + \gamma_{t}^2 \sigma^2 + \gamma_{t}^2 \rho L^2 V_t + \gamma_{t}^2 \rho L^2 \norm{r_t}^2 \nonumber\\
        &\leq
        \left(1 - \frac {5\gamma_{t} \mu}{6} + \gamma_{t}^2 \rho L^2\right) V_t
        +  \gamma_{t}^2 \sigma^2 + \gamma_{t}^2 \rho L^2 \norm{r_t}^2,
    \end{align}
since $1 - \gamma_t L \geq \frac{5}{6}$. Moreover, we have $ \gamma_{t} \leq \frac{\mu}{3\rho L^2} $, which means that $-\frac{5\gamma_{t} \mu }{6} + \gamma_{t}^2 \rho L^2 \leq -\frac{\gamma_{t} \mu}{2}$. By substituting this into \eqref{eq: recur-V}, we get
    \begin{align}
        \E [V_{t+1}]
        &\leq
        \left(1 - \frac {\gamma_{t} \mu}{2}\right) V_t
        + \gamma_{t}^2 \sigma^2 + \gamma_{t}^2 \rho L^2 \norm{r_t}^2.
        \label{eq:2204_1}
    \end{align}
Representing $t = kH + a \ (k, a \in \mathbb{N}; \ a < H)$, recalling that $V_{kH} = 0$, recursing~\eqref{eq:2204_1}, taking the full expectation and taking into account that the product of zero terms is equal to $1$, we conclude
    \begin{align*}
        \E [V_{t + 1}]
        \leq&\
        \summ{i=kH}{kH+a} \left(
        \rho \gamma_{i}^2 L^2 \E \norm{r_i}^2 + \gamma_{i}^2 \sigma^2
        \right)
        \prod\limits_{j = i + 1}^{kH+a}\left(1 - \frac{\gamma_{j} \mu}{2}\right)
        \\ &+ \prod\limits_{i = kH}^{kH+a}\left(1 - \frac{\gamma_{i} \mu}{2}\right)  \cdot V_{kH} \\
        =&\
        \summ{i=kH}{kH+a} \left(
        \rho \gamma_{i}^2 L^2 \E \norm{r_i}^2 + \gamma_{i}^2 \sigma^2
        \right)
        \prod\limits_{j = i + 1}^{kH+a}\left(1 - \frac{\gamma_{j} \mu}{2}\right). 
    \end{align*}
This ends the proof. \qed
\end{proof} 

%% file: stuff/supplementary/ProofTheorem1.tex
\begin{proof}
We choose tune stepsizes $\gamma_t$ as in Theorem 2 from \cite{woodworth2020local}. Proof is divided into two parts with two different techniques of choosing $\gamma_t$.

\noindent\textbf{Case (a)}: $T \leq 2\kappa$. 
In this option, we choose the constant stepsizes $\gamma_t = \gamma  = \frac{1}{6L}$ and the sequence of weights $w_t := (1 - \mu \gamma)^{-t-1}$.
Substituting the result of Lemmas \ref{lem:rho_2} and \ref{lem:lemma_Vt} into Lemma~\ref{lem:very_main}, we obtain
    \begin{align}
        \E \norm{\bar{x}_{t+1}-x_*}^2
            \leq&\
            (1 - \gamma \mu) \E \norm{\bar{x}_t - x_*}^2 
            + \frac{\gamma^2 \sigma^2}{M}
            - \frac{\gamma}{6} \E [F(\bar{x}_t) - F(x_*)] 
            \nonumber\\&+ 2 L_R (H-1) \gamma^3 \sigma^2.
    \end{align} \label{eq:interesting}
Rearranging,
    \begin{align}
        \frac{\gamma}{6} \E [F(\bar{x}_t) - F(x_*)]  \leq&\
            (1 - \gamma \mu) \E \norm{\bar{x}_t - x_*}^2 
            - \E \norm{\bar{x}_{t+1}-x_*}^2
            + \frac{\gamma^2 \sigma^2}{M}
            \nonumber \\&+ 2 L_R (H-1) \gamma^3 \sigma^2.
    \end{align}
Dividing both sides by $\gamma$,
    \begin{align*}
        \frac{1}{6} \E [F(\bar{x}_t) - F(x_*)]  \leq&\
            \frac{1}{\gamma}(1 - \gamma \mu) \E \norm{\bar{x}_t - x_*}^2 
            - \frac{1}{\gamma} \E \norm{\bar{x}_{t+1}-x_*}^2
            + \frac{\gamma \sigma^2}{M}
            \nonumber \\&+ 2 L_R (H-1) \gamma^2 \sigma^2.
    \end{align*}
Multiplying both sides by $w_t$ and summing from $0$ to $T$,
    \begin{align*}
        \frac{1}{6} \summ{t=0}{T-1} w_t \E [F(\bar{x}_t) - F(x_*)]
        \leq&\
        \summ{t=0}{T-1} w_t
        \Bigg( 
        \frac{1}{\gamma}(1 - \gamma \mu) \E \norm{\bar{x}_t - x_*}^2 
        - \frac{1}{\gamma} \E \norm{\bar{x}_{t+1}-x_*}^2
        \\ &+ \frac{\gamma \sigma^2}{M}
        + 2 L_R (H-1) \gamma^2 \sigma^2
        \Bigg) \\
        =&\
        \summ{t=0}{T-1}
        \Bigg(
        \frac{1}{\gamma}(1-\gamma \mu)^{-t} \E \norm{\bar{x}_{t}-x_*}^2
        \\&- \frac{1}{\gamma}(1-\gamma \mu)^{(-t + 1)} \E \norm{\bar{x}_{t+1}-x_*}^2
        \\&+ w_t 2 L_R (H-1) \gamma^2 \sigma^2 +  \frac{w_t\gamma \sigma^2}{M} \Bigg) \\
        \leq&\
        \frac{1}{\gamma } \E \norm{x_0 - x_*}^2 
        +
        \left(
        \frac{\gamma \sigma^2}{M}
        + 2 L_R (H-1) \gamma^2 \sigma^2
        \right)
        \summ{t=0}{T-1} w_t.
    \end{align*}
Let us denote $W_T := \summ{t=0}{T-1} w_t$. Hence, multiplying both parts by $\frac{1}{W_t}$, one can obtain
    \begin{align*}
        \frac{1}{6 W_T} \summ{t=0}{T} w_t \E [F(\bar{x}_t) - F(x_*)] 
        \leq 
        \frac{1}{\gamma W_T} \E \norm{x_0 - x_*}^2 
        + \frac{\gamma \sigma^2}{M}
        + 2 L_R (H-1) \gamma^2 \sigma^2.
    \end{align*}
Applying the Jensen's inequality \eqref{ap:Jensen} with $\tilde{x}_T := \frac{1}{W_T} \summ{t=0}{T} w_t \bar{x}_t$ and using the bound $W_T \geq (1 - \mu \gamma)^{-T} \geq \exp{(\mu \gamma T)}$, we have
    \begin{align*}
        \frac{1}{6} \E [F(\tilde{x}_T) - F(x_*)] 
        \leq 
        \frac{1}{\gamma} \exp{(-\mu \gamma T)} \norm{x_0 - x_*}^2
        + \frac{\gamma \sigma^2}{M}
        + 2 L_R (H-1) \gamma^2 \sigma^2.
    \end{align*}
Recalling \cref{st:st_1}, $\gamma = \frac{1}{6L}$ and $2L \geq \mu T$,
    \begin{align*}
        \E [F(\tilde{x}_T) - F(x_*)] 
        &\leq 
        36 L\exp{\left(-\frac{\mu T}{6L}\right)} \norm{x_0 - x_*}^2
        + \frac{2 \sigma^2}{\mu T M}
        + \frac{12 L_R (H-1) \sigma^2}{9\mu^2 T^2} \\
        &=
        \O \left( 
        L\exp{\left(-\frac{\mu T}{6 L}\right)} \norm{x_0 - x_*}^2
        + \frac{\sigma^2}{\mu T M}
        + \frac{\varepsilon L \sigma^2}{\mu^2 TK}
        \right),
    \end{align*}
which finishes the proof.

\noindent\textbf{Case (b)}: $T > 2\kappa$. Without restriction of generality, we can put $T$ even to get a simpler notation.
Let us choose the stepsize in the following way:
    \begin{align*}
    \gamma_t=
        \begin{cases}
             \frac{1}{6L} & \text{if } t \leq \frac{T}{2};
            \\
            \frac{2}{\mu(\xi + t)} & \text{if } t > \frac{T}{2},
        \end{cases}
    \end{align*}    
where $\xi = \frac{12 L}{\mu} - \frac{T}{2}$. It is noteworthy that $\gamma_t \leq \frac{1}{6L}$. Therefore, we can apply \cref{lem:lemma_Vt}.

\noindent Since $F(x) - F^* \geq 0$ for every $x$, we start as follows:
    \begin{align*}
        \E \norm{x_{T/2} - x_*}^2 
        \leq
        (1 - \gamma \mu) \E \norm{x_{T/2 - 1} - x_*}^2
        + \frac{\gamma^2 \sigma^2}{M}
        + 2 L_R (H-1) \gamma^3 \sigma^2.
    \end{align*}
Running the recursion for the relation above, we gain
    \begin{align}
        \E \norm{x_{T/2} - x_*}^2 
        &\leq
        (1 - \gamma \mu)^{T/2} \norm{x_{0} - x_*}^2
        + \frac{\gamma \sigma^2}{\mu M}
        + \frac{2 L_R (H-1) \gamma^2 \sigma^2}{\mu} \nonumber\\
        &=
        \left(1 - \frac{\mu}{6L}\right)^{T/2} \norm{x_{0} - x_*}^2
        + \frac{\sigma^2}{6L \mu M}
        + \frac{L_R (H-1) \sigma^2}{18 \mu L^2} \nonumber\\
        &\leq
        \left(1 - \frac{\mu}{6L}\right)^{T/2} \norm{x_{0} - x_*}^2
        + \frac{\sigma^2}{6L \mu M}
        + \frac{\varepsilon (H-1) \sigma^2}{18 \mu L} \nonumber\\
        &\leq
        \exp{\left(\frac{-\mu T}{12L}\right)} \norm{x_{0} - x_*}^2
        + \frac{\sigma^2}{6L \mu M}
        + \frac{\varepsilon (H-1) \sigma^2}{18 \mu L}, \label{eq: main}
    \end{align}
where we use that $\sum\limits_{i=0}^t (1 - \gamma \mu)^i \leq \frac{1}{\gamma \mu}$ and $L_R = \varepsilon L$. For the second part we apply similar technique as in the \textbf{case (a)}, but with $w_t = \xi + t$ and $W_T = \summ{t=T/2}{T-1} w_t$:
    \begin{align*}
        \frac{1}{6 W_T}& \summ{t=T/2}{T-1} w_t \E [F(\bar{x}_t) - F(x_*)] \nonumber\\
        \leq&\
        \frac{1}{W_T} \summ{t=T/2}{T-1}
        w_t \Bigg(
        \bigg(\frac{1}{\gamma_t} - \mu\bigg) \E \norm{x_t - x_*}^2
        - \frac{1}{\gamma_t} \E \norm{x_{t+1} - x_*}^2
        \nonumber\\&\hspace{5.7cm}+ \frac{\gamma \sigma^2}{M}
        + 2 L_R (H-1) \gamma^2 \sigma^2 
        \Bigg) \nonumber\\
        =&\
        \frac{1}{W_T} \summ{t=T/2}{T-1}
        \Bigg(
        \mu\bigg(\frac{\xi}{2} + \frac{t}{2}\bigg)(t + \xi - 2) \E \norm{x_{t} - x_*}^2
        - \frac{\mu (t+ \xi)^2}{2} \E \norm{x_{t+1} - x_*}^2
        \nonumber\\&\hspace{5.9cm}+ \frac{2 \sigma^2}{\mu M}
        + \frac{8 L_R (H-1) \sigma^2}{\mu^2 (\xi + t)}
        \Bigg) \nonumber\\
        \leq&\
        \frac{1}{W_T} \summ{t=T/2}{T-1}
        \Bigg(
        \frac{\mu (t + \xi - 1)^2}{2} \E \norm{x_{t} - x_*}^2
        - \frac{\mu(t + \xi)^2}{2} \E \norm{x_{t+1} - x_*}^2
        \nonumber\\&\hspace{5.9cm}+ \frac{2 \sigma^2}{\mu M}
        + \frac{8 L_R (H-1) \sigma^2}{\mu^2 (\xi + t)}
        \Bigg) \nonumber\\
        \leq&\
        \frac{\mu (T/2 + \xi - 1)^2}{W_T} \E \norm{x_{T/2} - x_*}^2
        + \frac{T \sigma^2}{W_T \mu M}
        + \frac{8 L_R (H-1) \sigma^2}{W_T\mu^2} \summ{t=T/2}{T-1} \frac{1}{\xi + t}.
    \end{align*}
Recalling the definition of $\xi$,
    \begin{align*}
        \summ{t=T/2}{T-1} \frac{1}{\xi + t} 
        &= 
        \summ{t=0}{T/2 - 1} \frac{1}{\left(t + \frac{12L}{\mu}\right)}
        \leq
        \summ{t=0}{T/2} \frac{1}{t + \frac{L}{\mu}} \leq \int\limits_{1}^{T/2 + L/\mu} \frac{dx}{x} + \frac{\mu}{L} \\&= \frac{\mu}{L} + \log\left(T/2 + \frac{L}{\mu}\right) \leq \frac{\mu}{L} + \log\left(T\right),
    \end{align*}
where we use $T \geq 2\kappa$. Therefore,
    \begin{align*}
        &\frac{1}{6 W_T} \summ{t=T/2}{T-1} w_t \E [F(\bar{x}_t) - F(x_*)] \\
        &\leq
        \frac{\mu (T/2 + \xi - 1)^2}{W_T} \E \norm{x_{T/2} - x_*}^2
        + \frac{T \sigma^2}{W_T \mu M}
        + \frac{8 L_R (H-1) \sigma^2}{W_T\mu^2} \left(\frac{\mu}{L} + \log\left(T\right)\right) \\
        &=
        \frac{\mu (12 L/\mu - 1)^2}{W_T} \E \norm{x_{T/2} - x_*}^2
        + \frac{T \sigma^2}{W_T \mu M}
        + \frac{8 L_R (H-1) \sigma^2}{W_T\mu^2}\left(\frac{\mu}{L} + \log\left(T\right)\right).
    \end{align*}
Observe that
    \begin{align}
    \label{eq: w_t_b}
        W_T = \summ{t=T/2}{T-1} (\xi + t) = \summ{t=0}{T/2-1} \left(\frac{12L}{\mu} + t\right) \geq \frac{T^2}{16}, 
    \end{align}
since, without loss of genetality, $T \geq 4$.
Thus, substituting the inequality above, we obtain:
    \begin{align}
        \frac{1}{6 W_T} \summ{t=T/2}{T-1} w_t \E [F(\bar{x}_t) - F(x_*)]
        \leq&\
        \frac{\mu(12L/\mu - 1)^2}{W_T} \E \norm{x_{T/2} - x_*}^2 
        + \frac{16 \sigma^2}{\mu T M}
        \nonumber\\&+ \frac{128 L_R (H-1) \sigma^2}{T^2 \mu^2} \left(\frac{\mu}{L} + \log\left(T\right)\right). \label{eq: wow_0}
    \end{align}
Applying the Jensen's inequality \eqref{ap:Jensen} with $\tilde{x}_T = \frac{1}{W_T} \summ{t=T/2}{T-1} w_t \bar{x_t}$:
    \begin{align}\label{eq: wow}
        F(\tilde{x}_T) \leq \frac{1}{W_T} \summ{t=T/2}{T} w_t F(\bar{x}_t).
    \end{align}
Using that $T \geq \frac{2L}{\mu}$ and combining \eqref{eq: main}, \eqref{eq: w_t_b}, \eqref{eq: wow_0} and \eqref{eq: wow}, we have
    \begin{align*}
        \E [F(\tilde{x}_t) - F(x_*)]
        \leq&\
        \frac{6\mu (12L/\mu - 1)^2}{W_T} \left(
        \exp{\left(\frac{-\mu T}{12L}\right)} \norm{x_{0} - x_*}^2 
        + \frac{\sigma^2}{6L \mu M}
        + \frac{\varepsilon (H-1) \sigma^2}{18 \mu L}
        \right) \\
        &+ \frac{96 \sigma^2}{\mu T M}
        + \frac{768 L_R (H-1) \sigma^2}{T^2 \mu^2} \left(\frac{\mu}{L} + \log\left(T\right)\right) \\
        \leq&\
        \frac{6\mu(12L/\mu)^2}{L^2 / 4\mu^2} 
        \exp{\left(\frac{-\mu T}{12L}\right)} \norm{x_{0} - x_*}^2
        + 
        \frac{(12L / \mu)^2 \sigma^2}{W_T L M}
        \\
        &+
        \frac{(12 L/\mu)^2 \varepsilon H \sigma^2}{3 W_T L} + \frac{96 \sigma^2}{\mu T M}
        + \frac{768 L_R (H-1) \sigma^2}{T^2 \mu^2}\left(\frac{\mu}{L} + \log\left(T\right)\right) \\
        \leq&\
        3456\mu \exp{\left(\frac{-\mu T}{12L}\right)} \norm{x_{0} - x_*}^2
        +\frac{12^2 \cdot 8 \sigma^2}{\mu TM} 
        +\frac{16^2 \cdot 3 \varepsilon L H \sigma^2}{\mu^2 T^2} \\
        &+ \frac{96 \sigma^2}{\mu T M}
        + \frac{768 \varepsilon L H \sigma^2}{T^2 \mu^2} \left(\frac{\mu}{L} + \log\left(T\right)\right) \\
        =&\
        \tilde{\O}
        \left(
        L \exp{\left(\frac{-\mu T}{L}\right)} \norm{x_{0} - x_*}^2
        +\frac{\sigma^2}{\mu TM}
        +\frac{\varepsilon L \sigma^2}{\mu^2 TK}
        \right),
    \end{align*}
which completes the proof.  \qed 
\end{proof}

%% file: stuff/supplementary/ProofTheorem2.tex
\begin{proof}
    Considering constant stepsizes $\gamma_t = \gamma \leq \frac{1}{6L}$, $\mu = 0$, $\rho = 0$ and applying results from Lemmas \ref{lem:rho_2} and \ref{lem:lemma_Vt} to \ref{lem:very_main}, we obtain:
    \begin{align*}
            \E \norm{\bar{x}_{t+1}-x_*}^2
            \leq
            \E \norm{\bar{x}_t - x_*}^2 
            + \frac{\gamma^2 \sigma^2}{M}
            - \frac{\gamma}{6} \E [F(\bar{x}_t) - F(x_*)] 
            + 2 \gamma L_R (H-1) \gamma^2 \sigma^2.
    \end{align*}
    Let us denote $r_t = \bar{x}_t - x_*$. Hence, rearranging the above equation yields:
    \begin{align*}
            \frac{\gamma}{6} \E [F(\bar{x}_t) - F(x_*)]
            \leq
            \E \norm{r_t}^2 
            - \E \norm{r_{t+1}}^2
            + \frac{\gamma^2 \sigma^2}{M}
            + 2 L_R (H-1) \gamma^3 \sigma^2.
    \end{align*}
Averaging over $t$,
    \begin{align}
            \frac{\gamma}{6 T} \sum_{t=0}^{T-1} \E [F(\bar{x}_t) - F(x_*)]
            \leq&\ 
            \frac{1}{T} \sum_{t=0}^{T-1} 
            \left(\E \norm{r_t}^2 
            - \E \norm{r_{t+1}}^2\right)
            + \frac{\gamma^2 \sigma^2}{M}
            \nonumber\\&+ 2 L_R (H-1) \gamma^3 \sigma^2 \nonumber\\
            =&\ \frac{\norm{r_0}^2 - \E\norm{r_T}^2}{T}
            + \frac{\gamma^2 \sigma^2}{M}
            + 2 L_R (H-1) \gamma^3 \sigma^2  \nonumber\\
            \leq&\
            \frac{\norm{r_0}^2}{T}
            + \frac{\gamma^2 \sigma^2}{M}
            + 2 L_R (H-1) \gamma^3 \sigma^2. \label{eq:avergingovert}
    \end{align}
    For $\hat{x}_T = \frac{1}{T} \sum_{t=0}^{T-1} \bar{x}_t$, let us apply the Jensen's inequality \eqref{ap:Jensen}:
    \begin{align}
            \E [F(\hat{x}_T) - F(x_*)] \leq \frac{1}{T} \sum_{t=0}^{T-1} \E [F(\bar{x}_t) - F(x_*)]. \label{eq:jens}
    \end{align}
Plugging \eqref{eq:jens} into \eqref{eq:avergingovert},
    \begin{align*}
            \frac{\gamma}{6} \E [F(\hat{x}_T) - F(x_*)]
            &\leq \frac{\norm{r_0}^2}{T}
            + \frac{\gamma^2 \sigma^2}{M}
            + 2 L_R (H-1) \gamma^3 \sigma^2.
    \end{align*}
Dividing both sides by $\frac{\gamma}{6}$, we have:
    \begin{align*}
            \E [F(\hat{x}_T) - F(x_*)]
            &\leq \frac{6}{ \gamma T} \norm{r_0}^2
            + \frac{6 \gamma \sigma^2}{M}
            + 12 L_R (H-1) \gamma^2 \sigma^2.
    \end{align*}
Now, choosing $\gamma$ as follows:
    \begin{align*}
    \gamma = 
        \begin{cases}
            \min \left\{  \frac{1}{6L}, \frac{\norm{r_0} \sqrt{M}}{\sigma \sqrt{T}}  \right\},\qquad &\text{ if } H = 1 \text{ or } M = 1, \\
            \min \left\{ \frac{1}{6L}, \frac{\norm{r_0} \sqrt{M}}{\sigma \sqrt{T}}, \left(\frac{\norm{r_0}^2}{L_R \sigma^2TH} \right)^{1/3} \right\}, \qquad &\text { else, }
        \end{cases}
    \end{align*}
we conclude
    \begin{align*}
        \E [F(\hat{x}_T) - F(x_*)] 
        \leq&\
        \max \left\{ \frac{36L \norm{r_0}^2}{T}, \frac{6\sigma \norm{r_0}}{\sqrt{MT}},
        6\left( \frac{L_R \sigma^2 \norm{r_0}^4}{TK} \right)^{1/3}
        \right\} \\
        &+\ 
        \frac{6 \sigma \norm{r_0}}{\sqrt{MT}}
        +
        12 \cdot \left( \frac{L_R \sigma^2 \norm{r_0}^4}{TK} \right)^{1/3} \\
        =&\
        \O \left(
        \frac{L \norm{r_0}^2}{T} 
        + \frac{\sigma \norm{r_0}}{\sqrt{MT}} 
        + \left( \frac{\varepsilon L_R \sigma^2 \norm{r_0}^4}{TK} \right)^{1/3}
        \right).
    \end{align*}
This finalizes the proof. \qed
\end{proof}

%% file: stuff/supplementary/ProofTheorem3.tex
\begin{proof}
    We start with substituting the result of \cref{lem:rho_2} into \cref{lem:very_main}:
\begin{align*}
    \frac{\gamma_{t}}{6} \E [F(\bar{x}_t) - F(x_*)]
    \leq& (1 - \gamma_{t} \mu) \E \norm{r_t}^2 
    + \gamma_{t}^2 \E \norm{ \bar{\mc{g_t}} - \bar{g}_t}^2 - \E \norm{r_{t+1}}^2
    + 2 \gamma_{t} L_R \E [V_t] \\
    \leq&\
    (1 - \gamma_{t} \mu) \E \norm{r_t}^2 - \E \norm{r_{t+1}}^2
    + 2 \gamma_{t} L_R \E [V_t]\\&+ \gamma_{t}^2 \left(  \frac{\sigma^2}{M} + \frac{\rho L^2 }{M} \E[V_t] + \frac{\rho L^2 }{M}\E \norm{r_t}^2  \right).
\end{align*}
Suppose that $\gamma_{t}$ satisfies the assumption of \cref{lem:lemma_Vt} for the case $(d)$: $\gamma_t \leq \min\left\{ \frac{\mu}{3\rho L^2}, \frac{1}{6L} \right\}$. Therefore, with $t = kH + a$,
\begin{align}
\label{eq: main-prob}
    \frac{\gamma_{t}}{6} &\E [F(\bar{x}_t) - F(x_*)]\nonumber\\ \leq&\ (1 - \gamma_{t} \mu) \E \norm{r_t}^2 - \E \norm{r_{t+1}}^2
    + 2 \gamma_{t} L_R \E [V_t] + \gamma_{t}^2 \left(  \frac{\sigma^2}{M} + \frac{\rho L^2 }{M} V_t + \frac{\rho L^2 }{M}\E \norm{r_t}^2  \right)\nonumber\\
    =&\
    \left(1 - \gamma_{t} \mu + \frac{\gamma^2_t\rho L^2 }{M} \right) \E \norm{r_t}^2 - \E \norm{r_{t+1}}^2
    + \frac{\gamma_{t}^2\sigma^2}{M} + \left(\frac{\gamma_{t}^2 \rho L^2 }{M} + 2 \gamma_{t} L_R\right)\E[V_t]\nonumber \\
    \leq&\
    \left(1 - \frac{2\gamma_{t} \mu}{3}\right) \E \norm{r_t}^2 -\E \norm{r_{t+1}}^2+ \frac{\gamma_{t}^2\sigma^2}{M}
    + \left(\frac{\gamma_t\rho L^2}{M} + 2L_R\right)  \nonumber\\ &\times\gamma_t\left(\summ{i=kH}{kH+a - 1} \left(
                \rho \gamma_{i}^2 L^2 \E \norm{r_i}^2 + \gamma_{i}^2 \sigma^2\right)
                \prod\limits_{j = i + 1}^{kH+a - 1}\left(1 - \frac{\gamma_{j} \mu}{2}\right)\right),
\end{align}
where we use that $\gamma_t \leq \frac{\mu}{3\rho L^2}$. Next, assume that $\gamma_t \equiv \gamma$ and denote some parameters for the convenient reformulation of \eqref{eq: main-prob}:
\begin{equation}
    \label{eq: parameters}
    \begin{split}
        \beta_1 &:= 1 - \frac{2\gamma \mu}{3};\\
        \beta_2 &:= 1 - \frac{\gamma \mu}{2};\\
        C_0 &:= \frac{\gamma\rho L^2}{M} + 2 L_R;\\
        C &:= \frac{\mu}{3M} + 2 L_R \geq C_0;  
    \end{split}
\end{equation}
As a consequence, with \eqref{eq: parameters} we obtain
\begin{align}
    \label{eq: main-reform}
    \frac{1}{6} \E [F(\bar{x}_t) - F(x_*)] \leq&\
    \frac{\beta_1}{\gamma}\E \norm{r_t}^2 - \frac{1}{\gamma}\E \norm{r_{t+1}}^2+ \frac{\gamma\sigma^2}{M}
    +  \nonumber\\&+C_0\left(\summ{i=kH}{t - 1} \left(
                \rho \gamma^2 L^2 \E \norm{r_i}^2 + \gamma^2 \sigma^2\right)
                \beta_2^{t - i -1}\right).
\end{align}
The technique applied above is the same as in previous theorems. Nevertheless, we start with considering of only slice of total iterations to be summarized. That is, assume that $t \in \overline{nH, (n+1)H - 1}$ for some $n \in \mathbb{N}$. Accordingly, summing with weights $w_t$,
\begin{align*}
    \sum\limits_{t = nH}^{(n+1)H-1}\frac{w_t}{6}\E [F(\bar{x}_t) - F(x_*)] \leq& \sum\limits_{t = nH}^{(n+1)H-1} \Bigg(\frac{w_t\beta_1}{\gamma}\E \norm{r_t}^2 - \frac{w_t}{\gamma}\E \norm{r_{t+1}}^2+ \frac{w_t\gamma\sigma^2}{M}
    \nonumber\\+&\ C_0 w_t\left(\summ{i=nH}{t - 1} \left(
                \rho \gamma^2 L^2 \E \norm{r_i}^2 + \gamma^2 \sigma^2\right)
                \beta_2^{t - i -1}\right)\Bigg).
\end{align*}
Let us find the coefficient of $\E\norm{r_p}^2$ for $p \in \overline{nH, (n+1)H}$ in the inequality above. At first, start with $p \in \overline{nH + 1, (n+1)H - 2}$. After opening the brackets, it turns out that it is equal to
\begin{align}
    \label{eq: coef}
    \frac{w_p\beta_1}{\gamma} - \frac{w_{p-1}}{\gamma} + C_0\rho \gamma^2L^2\beta_2^{-p -1}\left(\sum\limits_{k=p+1}^{(n+1)H-1}w_k\beta_2^k \right).
\end{align}
Now, let us choose $w_t := \left(1 - \frac{\gamma\mu}{2}\right)^{-(t+1)}$. Therefore, the multiplicative factor can be bounded from above as
\begin{align}
    \label{eq: coef_r}
    \frac{w_p\beta_1}{\gamma} - \frac{w_{p-1}}{\gamma} + \frac{C\rho \gamma^2L^2Hw_p}{1 - \frac{\gamma\mu}{2}},
\end{align}
since $(n+1)H - p - 1 \leq H$. Multiplying both sides by $\left(1 - \frac{\gamma\mu}{2}\right)^{p+2}$ and substituting \eqref{eq: parameters}, we get
\begin{align*}
    \frac{1}{\gamma}\left(-\frac{\gamma\mu}{6}\left(1 - \frac{\gamma\mu}{2}\right)\right) + C\rho \gamma^2L^2H \leq 0.
\end{align*}
Solving the quadratic inequality and applying the fact that $\sqrt{a^2 + b^2} \leq a + b$, one can obtain
\begin{align}
    \label{eq: step-recur-choice}
    \gamma \leq \frac{\sqrt{\mu}}{\sqrt{6CH\rho L^2}}.
\end{align}
If we choose $\gamma$ according to \eqref{eq: step-recur-choice}, then for all $n$ and for all \\ $p: p \in \overline{nH + 1, (n+1)H - 2}$ the coefficient at $\E\norm{r_p}^2$ is less or equal to $0$. Moreover, the case $p = (n+1)H - 1$ is obvious since the factor is equal to 
\begin{align*}
    \frac{w_p\beta_1}{\gamma} - \frac{w_{p-1}}{\gamma}.
\end{align*}
Because of \eqref{eq: coef_r} it is less than $0$. The last case is when $p$ is the moment of communication. Since for the final convergence it is necessary to summarize all terms responsible for the period between communications, then for the option $p \mod H = 0$ the coefficient at $\E\norm{r_p}^2$ is equal to the sum of two terms: one part appears in summing from $nH$ to $(n+1)H-1$ and another one is in $(n-1)H$ to $nH-1$. However, it is equal to \eqref{eq: coef} with $p=nH$. Applying the bound  \eqref{eq: coef_r}, we claim that the factor is also less or equal to $0$. Consequently, with $T = KH$
\begin{align*}
    \sum\limits_{n=0}^{K-1}\sum\limits_{t = nH}^{(n+1)H-1}\frac{w_t}{6}\E [F(\bar{x}_t) - F(x_*)] \leq& \norm{r_0}^2\left(\frac{w_0}{\gamma}\beta_1 + C\rho\gamma^2 L^2w_0 \sum\limits_{k=1}^{H-1}\beta_2^kw_k\right)\nonumber\\&-\E\norm{r_T}^2\left(\frac{w_{T-1}}{\gamma}\right) \nonumber\\&+ \sum\limits_{n=0}^{K-1}\sum\limits_{t = nH}^{(n+1)H-1} C_0w_t\gamma^2\sigma^2\sum\limits_{i=nH}^{t-1}\beta_2^{t-i-1} \nonumber\\&+\sum\limits_{t = 0}^{T-1} \frac{w_t\gamma\sigma^2}{M} \nonumber \\\leq&\ 
    \norm{r_0}^2\left(\frac{w_0}{\gamma}\beta_1 + C\rho\gamma^2 L^2w_0 \sum\limits_{k=1}^{H-1}\beta_2^kw_k\right)\nonumber\\&+\sum\limits_{t = 0}^{T-1} \left(\frac{w_t\gamma\sigma^2}{M} + C_0w_tH\gamma^2\sigma^2\right). 
\end{align*}
Therefore, using that $\beta_1 \leq \beta_2$ and $\beta_2^k w_k = \frac{1}{1 -\frac{\gamma\mu}{2}}$, we have
\begin{align}
    \label{eq: final-333}
    \sum\limits_{n=0}^{K-1}\sum\limits_{t = nH}^{(n+1)H-1}\frac{w_t}{6}\E [F(\bar{x}_t) - F(x_*)] \leq&\ 
    \norm{r_0}^2\left(\frac{1}{\gamma} + \frac{CH\rho\gamma^2 L^2w_0}{1 -\frac{\gamma\mu}{2}}\right)\nonumber\\&+\sum\limits_{t = 0}^{T-1} \left(\frac{w_t\gamma\sigma^2}{M} + C_0w_tH\gamma^2\sigma^2\right) \nonumber \\\leq&\ 
    \norm{r_0}^2\left(\frac{1}{\gamma} + \frac{24 \mu}{121}\right)\nonumber\\&+\sum\limits_{t = 0}^{T-1} \left(\frac{w_t\gamma\sigma^2}{M} + C_0w_tH\gamma^2\sigma^2\right),
\end{align}
since $\gamma \leq \frac{1}{6L}$ and $\gamma \leq \frac{\sqrt{\mu}}{\sqrt{6CH\rho L^2}}$. Dividing both sides by $W_t = \sum\limits_{t=0}^{T-1}w_t$ and applying the Jensen's inequality \eqref{ap:Jensen} to \eqref{eq: final-333}, one can obtain
\begin{align*}
    \E [F(\tilde{x}_t) - F(x_*)] \leq \frac{6\norm{r_0}^2}{W_T}\left(\frac{1}{\gamma} + \frac{24\mu}{121}\right) + \frac{6\gamma\sigma^2}{M} + 6C_0H\gamma^2\sigma^2.
\end{align*}
Choosing $\gamma$ as $\min\left\{\frac{\mu}{3\rho L^2}, \frac{1}{6L}, \frac{\sqrt{\mu}}{\sqrt{6CH\rho L^2}}, \frac{\ln\left(\max\left\{2, \mu^2\norm{r_0}^2T^2M/\sigma^2\right\}\right)}{2\mu T}\right\}$ and using that $W_T \geq w_{T-1} = \left(1 - \frac{\gamma\mu}{2}\right)^{-T} \geq \exp\left(\frac{\gamma\mu T}{2}\right)$, we get that if $\gamma$ is equal to the first, second or third options, then
\begin{align*}
    \E [F(\tilde{x}_t) - F(x_*)] =&\ \tilde{\O}\Bigg(\norm{r_0}^2\max\left\{\frac{\rho L^2}{\mu}, L, \frac{\sqrt{(\mu + L_R)H\rho L^2}}{\sqrt{\mu}}\right\}\\&\times\exp\left(-\mu T \min\left\{\frac{\mu}{\rho L^2}, \frac{1}{L}, \frac{\sqrt{\mu}}{\sqrt{(\mu + L_R)H\rho L^2}}\right\}\right) \\&+ \frac{\sigma^2}{\mu M T} + \frac{\rho L^2H\sigma^2}{M\mu^3T^3}+  \frac{L_RH\sigma^2}{\mu^2 T^2}\Bigg),
\end{align*}
where we substitute $C_0$ from \eqref{eq: parameters}. Moreover, the fourth choice of $\gamma$ gives
\begin{align*}
    \E [F(\tilde{x}_t) - F(x_*)] =\tilde{\O}\Bigg(\frac{\sigma^2}{\mu M T} + \frac{\rho L^2H\sigma^2}{M\mu^3T^3}+  \frac{L_RH\sigma^2}{\mu^2 T^2}\Bigg).
\end{align*}
Combining two upper bound above we finish the proof. \qed
\end{proof}

%% file: main.bbl
\begin{thebibliography}{10}
\providecommand{\url}[1]{\texttt{#1}}
\providecommand{\urlprefix}{URL }
\providecommand{\doi}[1]{https://doi.org/#1}

\bibitem{abadi2016deep}
Abadi, M., Chu, A., Goodfellow, I., McMahan, H.B., Mironov, I., Talwar, K., Zhang, L.: Deep learning with differential privacy. In: Proceedings of the 2016 ACM SIGSAC conference on computer and communications security. pp. 308--318 (2016)

\bibitem{basu2019qsparse}
Basu, D., Data, D., Karakus, C., Diggavi, S.: Qsparse-local-sgd: Distributed sgd with quantization, sparsification and local computations. Advances in Neural Information Processing Systems  \textbf{32} (2019)

\bibitem{beznosikov2022decentralized}
Beznosikov, A., Dvurechenskii, P., Koloskova, A., Samokhin, V., Stich, S.U., Gasnikov, A.: Decentralized local stochastic extra-gradient for variational inequalities. Advances in Neural Information Processing Systems  \textbf{35},  38116--38133 (2022)

\bibitem{beznosikov2020distributed}
Beznosikov, A., Samokhin, V., Gasnikov, A.: Distributed saddle-point problems: Lower bounds, near-optimal and robust algorithms. arXiv preprint arXiv:2010.13112  (2020)

\bibitem{beznosikov2021distributed}
Beznosikov, A., Scutari, G., Rogozin, A., Gasnikov, A.: Distributed saddle-point problems under data similarity. Advances in Neural Information Processing Systems  \textbf{34},  8172--8184 (2021)

\bibitem{beznosikov2024similarity}
Beznosikov, A., Tak{\'a}c, M., Gasnikov, A.: Similarity, compression and local steps: three pillars of efficient communications for distributed variational inequalities. Advances in Neural Information Processing Systems  \textbf{36} (2024)

\bibitem{chezhegov2024local}
Chezhegov, S., Skorik, S., Khachaturov, N., Shalagin, D., Avetisyan, A., Beznosikov, A., Tak{\'a}{\v{c}}, M., Kholodov, Y., Gasnikov, A.: Local methods with adaptivity via scaling. arXiv preprint arXiv:2406.00846  (2024)

\bibitem{ghadimi2013optimal}
Ghadimi, S., Lan, G.: Optimal stochastic approximation algorithms for strongly convex stochastic composite optimization, ii: shrinking procedures and optimal algorithms. SIAM Journal on Optimization  \textbf{23}(4),  2061--2089 (2013)

\bibitem{glasgow2022sharp}
Glasgow, M.R., Yuan, H., Ma, T.: Sharp bounds for federated averaging (local sgd) and continuous perspective. In: International Conference on Artificial Intelligence and Statistics. pp. 9050--9090. PMLR (2022)

\bibitem{goodfellow2016deep}
Goodfellow, I., Bengio, Y., Courville, A.: Deep learning. MIT press (2016)

\bibitem{gorbunov2021local}
Gorbunov, E., Hanzely, F., Richt{\'a}rik, P.: Local sgd: Unified theory and new efficient methods. In: International Conference on Artificial Intelligence and Statistics. pp. 3556--3564. PMLR (2021)

\bibitem{hendrikx2020statistically}
Hendrikx, H., Xiao, L., Bubeck, S., Bach, F., Massoulie, L.: Statistically preconditioned accelerated gradient method for distributed optimization. In: International conference on machine learning. pp. 4203--4227. PMLR (2020)

\bibitem{kairouz2021advances}
Kairouz, P., McMahan, H.B., Avent, B., Bellet, A., Bennis, M., Bhagoji, A.N., Bonawitz, K., Charles, Z., Cormode, G., Cummings, R., et~al.: Advances and open problems in federated learning. Foundations and trends{\textregistered} in machine learning  \textbf{14}(1--2),  1--210 (2021)

\bibitem{karimireddy2020scaffold}
Karimireddy, S.P., Kale, S., Mohri, M., Reddi, S., Stich, S., Suresh, A.T.: Scaffold: Stochastic controlled averaging for federated learning. In: International conference on machine learning. pp. 5132--5143. PMLR (2020)

\bibitem{khaled2020tighter}
Khaled, A., Mishchenko, K., Richt{\'a}rik, P.: Tighter theory for local sgd on identical and heterogeneous data. In: International Conference on Artificial Intelligence and Statistics. pp. 4519--4529. PMLR (2020)

\bibitem{koloskova2020unified}
Koloskova, A., Loizou, N., Boreiri, S., Jaggi, M., Stich, S.: A unified theory of decentralized sgd with changing topology and local updates. In: International Conference on Machine Learning. pp. 5381--5393. PMLR (2020)

\bibitem{konevcny2016federated}
Kone{\v{c}}n{\`y}, J., McMahan, H.B., Yu, F.X., Richt{\'a}rik, P., Suresh, A.T., Bacon, D.: Federated learning: Strategies for improving communication efficiency. arXiv preprint arXiv:1610.05492  (2016)

\bibitem{kovalev2022optimal}
Kovalev, D., Beznosikov, A., Borodich, E., Gasnikov, A., Scutari, G.: Optimal gradient sliding and its application to optimal distributed optimization under similarity. Advances in Neural Information Processing Systems  \textbf{35},  33494--33507 (2022)

\bibitem{li2020federated}
Li, T., Sahu, A.K., Zaheer, M., Sanjabi, M., Talwalkar, A., Smith, V.: Federated optimization in heterogeneous networks. Proceedings of Machine learning and systems  \textbf{2},  429--450 (2020)

\bibitem{liang2019variance}
Liang, X., Shen, S., Liu, J., Pan, Z., Chen, E., Cheng, Y.: Variance reduced local sgd with lower communication complexity. arXiv preprint arXiv:1912.12844  (2019)

\bibitem{mangasarian1995parallel}
Mangasarian, L.: Parallel gradient distribution in unconstrained optimization. SIAM Journal on Control and Optimization  \textbf{33}(6),  1916--1925 (1995)

\bibitem{mcmahan2017communication}
McMahan, B., Moore, E., Ramage, D., Hampson, S., y~Arcas, B.A.: Communication-efficient learning of deep networks from decentralized data. In: Artificial intelligence and statistics. pp. 1273--1282. PMLR (2017)

\bibitem{mishchenko2022proxskip}
Mishchenko, K., Malinovsky, G., Stich, S., Richt{\'a}rik, P.: Proxskip: Yes! local gradient steps provably lead to communication acceleration! finally! In: International Conference on Machine Learning. pp. 15750--15769. PMLR (2022)

\bibitem{reisizadeh2020fedpaq}
Reisizadeh, A., Mokhtari, A., Hassani, H., Jadbabaie, A., Pedarsani, R.: Fedpaq: A communication-efficient federated learning method with periodic averaging and quantization. In: International conference on artificial intelligence and statistics. pp. 2021--2031. PMLR (2020)

\bibitem{robbins1951stochastic}
Robbins, H., Monro, S.: A stochastic approximation method. The annals of mathematical statistics pp. 400--407 (1951)

\bibitem{schmidt2013fast}
Schmidt, M., Roux, N.L.: Fast convergence of stochastic gradient descent under a strong growth condition. arXiv preprint arXiv:1308.6370  (2013)

\bibitem{shalev2014understanding}
Shalev-Shwartz, S., Ben-David, S.: Understanding machine learning: From theory to algorithms. Cambridge university press (2014)

\bibitem{shamir2014communication}
Shamir, O., Srebro, N., Zhang, T.: Communication-efficient distributed optimization using an approximate newton-type method. In: International conference on machine learning. pp. 1000--1008. PMLR (2014)

\bibitem{sharma2019parallel}
Sharma, P., Kafle, S., Khanduri, P., Bulusu, S., Rajawat, K., Varshney, P.K.: Parallel restarted spider--communication efficient distributed nonconvex optimization with optimal computation complexity. arXiv preprint arXiv:1912.06036  (2019)

\bibitem{spiridonoff2021communication}
Spiridonoff, A., Olshevsky, A., Paschalidis, Y.: Communication-efficient sgd: From local sgd to one-shot averaging. Advances in Neural Information Processing Systems  \textbf{34},  24313--24326 (2021)

\bibitem{stich2018local}
Stich, S.U.: Local sgd converges fast and communicates little. arXiv preprint arXiv:1805.09767  (2018)

\bibitem{verbraeken2020survey}
Verbraeken, J., Wolting, M., Katzy, J., Kloppenburg, J., Verbelen, T., Rellermeyer, J.S.: A survey on distributed machine learning. Acm computing surveys (csur)  \textbf{53}(2),  1--33 (2020)

\bibitem{wang2019slowmo}
Wang, J., Tantia, V., Ballas, N., Rabbat, M.: Slowmo: Improving communication-efficient distributed sgd with slow momentum. arXiv preprint arXiv:1910.00643  (2019)

\bibitem{woodworth2020local}
Woodworth, B., Patel, K.K., Stich, S., Dai, Z., Bullins, B., Mcmahan, B., Shamir, O., Srebro, N.: Is local sgd better than minibatch sgd? In: International Conference on Machine Learning. pp. 10334--10343. PMLR (2020)

\bibitem{yuan2020federated}
Yuan, H., Ma, T.: Federated accelerated stochastic gradient descent. Advances in Neural Information Processing Systems  \textbf{33},  5332--5344 (2020)

\bibitem{zhang2016parallel}
Zhang, J., De~Sa, C., Mitliagkas, I., R{\'e}, C.: Parallel sgd: When does averaging help? arXiv preprint arXiv:1606.07365  (2016)

\end{thebibliography}
